\documentclass[12pt,oneside]{amsart}
\usepackage[latin9]{inputenc}
\usepackage{mathtools}
\usepackage{amstext}
\usepackage{amsthm}
\usepackage{amssymb}
\usepackage[unicode=true,
 bookmarks=false,
 breaklinks=false,pdfborder={0 0 1},backref=section,colorlinks=false]
 {hyperref}
\usepackage{breakurl}

\makeatletter
\numberwithin{equation}{section}
\numberwithin{figure}{section}
\theoremstyle{plain}
\newtheorem{thm}{\protect\theoremname}
  \theoremstyle{plain}
  \newtheorem{cor}[thm]{\protect\corollaryname}
  \theoremstyle{plain}
  \newtheorem{prop}[thm]{\protect\propositionname}
  \theoremstyle{plain}
  \newtheorem{lem}[thm]{\protect\lemmaname}
  \theoremstyle{remark}
  \newtheorem{rem}[thm]{\protect\remarkname}

\@ifundefined{date}{}{\date{}}

\usepackage{latexsym}
\usepackage{amscd}
\usepackage{enumerate}
\usepackage{mathabx}

\newtheorem{question}[thm]{Question}{\bf}
\newcommand{\op}{\operatorname}

  \providecommand{\corollaryname}{Corollary}
  \providecommand{\lemmaname}{Lemma}
  \providecommand{\propositionname}{Proposition}
  \providecommand{\remarkname}{Remark}
\providecommand{\theoremname}{Theorem}

  \providecommand{\corollaryname}{Corollary}
  \providecommand{\lemmaname}{Lemma}
  \providecommand{\propositionname}{Proposition}
  \providecommand{\remarkname}{Remark}
\providecommand{\theoremname}{Theorem}

  \providecommand{\corollaryname}{Corollary}
  \providecommand{\lemmaname}{Lemma}
  \providecommand{\propositionname}{Proposition}
  \providecommand{\remarkname}{Remark}
\providecommand{\theoremname}{Theorem}

\makeatother

  \providecommand{\corollaryname}{Corollary}
  \providecommand{\lemmaname}{Lemma}
  \providecommand{\propositionname}{Proposition}
  \providecommand{\remarkname}{Remark}
\providecommand{\theoremname}{Theorem}

\begin{document}

\title{Sub-leading asymptotics of ECH capacities}

\author{Dan Cristofaro-Gardiner and Nikhil Savale}

\address{Department of Mathematics, University of California Santa Cruz, CA
95064, United States}

\email{dcristof@ucsc.edu}

\address{Universität zu Köln, Mathematisches Institut, Weyertal 86-90, 50931
Köln, Germany}

\email{nsavale@math.uni-koeln.de}

\thanks{D. C-G. is partially supported by NSF grant 1711976.}

\thanks{N. S. is partially supported by the DFG funded project CRC/TRR 191.}

\begin{abstract}
In previous work \cite{Cristofaro-Gardiner-Hutchings-Gripp2015},
the first author and collaborators showed that the leading asymptotics
of the embedded contact homology (ECH) spectrum recovers the contact
volume. Our main theorem here is a new bound on the sub-leading asymptotics. 
\end{abstract}

\maketitle

\section{Introduction}

\subsection{The main theorem}

Let $Y$ be a closed, oriented three-manifold. A \textit{contact form}
on $Y$ is a one-form $\lambda$ satisfying 
\[
\lambda\wedge d\lambda>0.
\]
A contact form determines the \textit{Reeb vector field}, $R$, defined
by 
\[
\lambda(R)=1,\quad d\lambda(R,\cdot)=0,
\]
and the \textit{contact structure} $\xi\coloneqq\textrm{Ker}(\lambda).$
Closed orbits of $R$ are called \textit{Reeb orbits}.

If $(Y,\lambda)$ is a closed three-manifold equipped with a nondegenerate
contact form and $\Gamma\in H_{1}(Y)$, then the \textit{embedded
contact homology} $ECH(Y,\lambda,\Gamma)$ is defined. This is the
homology of a chain complex freely generated over $\mathbb{Z}_{2}$
by certain sets of Reeb orbits in the homology class $\Gamma$, relative
to a differential that counts certain $J$-holomorphic curves in $\mathbb{R}\times Y.$
(ECH can also be defined over $\mathbb{Z}$, but for the applications
in this paper we will not need this.) It is known that the homology
only depends on $\xi$, and so we sometimes denote it $ECH(Y,\lambda,\xi)$.
Any Reeb orbit $\gamma$ has a \textit{symplectic action} 
\[
\mathcal{A}(\gamma)=\int_{\gamma}\lambda
\]
and this induces a filtration on $ECH(Y,\lambda)$; we can use this
filtration to define a number $c_{\sigma}(Y,\lambda)$ for every nonzero
class in ECH, called the \textit{spectral invariant} associated to
$\sigma$; the spectral invariants are $C^{0}$ continuous and so
can be extended to degenerate contact forms as well. We will review
the definition of ECH and of the spectral invariants in \S\ref{subsec:Embedded-contact-homology}.

When the class $c_{1}(\xi)+2\textrm{P.D.}(\Gamma)\in H^{2}(Y;\mathbb{Z})$
is torsion, then $ECH(Y,\xi,\Gamma)$ has a relative $\mathbb{Z}$
grading, which we can refine to a canonical absolute grading $\textrm{gr}^{\mathbb{Q}}$
by rationals \cite{Kronheimer-Mrowka}, and which we will review in
\S\ref{subsec:ECH=00003D00003D00003D00003D00003DHM}. It is known
that for large gradings the group is eventually 2-periodic and non-vanishing:
\[
ECH_{*}(Y,\xi,\Gamma)=ECH_{*+2}(Y,\xi,\Gamma)\neq0,\;\text{\ensuremath{\ast}}\gg0.
\]
The main theorem of \cite{Cristofaro-Gardiner-Hutchings-Gripp2015}
states that in this case, the asymptotics of the spectral invariants
recover the contact volume 
\[
\textrm{vol}(Y,\lambda)=\int_{Y}\lambda\wedge d\lambda.
\]
Specifically: 
\begin{thm}
\cite[Thm. 1.3]{Cristofaro-Gardiner-Hutchings-Gripp2015} \label{thm:vc}
Let $(Y,\lambda)$ be a closed, connected oriented three-manifold
with a contact form, and let $\Gamma\in H_{1}(Y)$ be such that $c_{1}(\xi)+2\textrm{P.D.}(\Gamma)$
is torsion. Then if $\lbrace\sigma_{j}\rbrace$ is any sequence of
nonzero classes in $ECH(Y,\xi,\Gamma)$ with definite gradings tending
to positive infinity, 
\begin{equation}
\lim_{j\rightarrow\infty}\frac{c_{\sigma_{j}}(Y,\lambda)^{2}}{\textrm{gr}^{\mathbb{Q}}(\sigma_{j})}=\op{vol}(Y,\lambda).\label{eqn:asymptoticformula}
\end{equation}
\end{thm}
The formula \eqref{eqn:asymptoticformula} has had various implications
for dynamics. For example, it was a crucial ingredient in recent work
\cite{Cristofaro-Gardiner-Hutchings-Pomerleano-2017} of the first
author and collaborators showing that many Reeb vector fields on closed
three-manifolds have either two or infinitely many distinct closed
orbits, and it was used in \cite{Cristofaro-Gardiner-Hutchings2016}
to show that every Reeb vector field on a closed three-manifold has
at least two distinct closed orbits. It has also been used to prove
$C^{\infty}$ closing lemmas for Reeb flows on closed three-manifolds
and Hamiltonian flows on closed surfaces \cite{Asaoka-Irie2016,Irie2015}.

By \eqref{eqn:asymptoticformula}, we can write 
\begin{equation}
c_{\sigma_{j}}(Y,\lambda)=\sqrt{\textrm{vol}(Y,\lambda)\cdot\textrm{gr}^{\mathbb{Q}}(\sigma_{j})}+d(\sigma_{j}),\label{eqn:defndk}
\end{equation}
where $d(\sigma_{j})$ is $o(\textrm{gr}^{\mathbb{Q}}(\sigma_{j})^{1/2})$
as $\textrm{gr}^{\mathbb{Q}}(\sigma_{j})$ tends to positive infinity.
It is then natural to ask:

\begin{question} What can we say about the asymptotics of $d(\sigma_{j})$
as $\textrm{gr}^{\mathbb{Q}}(\sigma_{j})$ tends to positive infinity?
\end{question}

Previously, W. Sun has shown that $d(\sigma_{j})$ is $O(\textrm{gr}^{\mathbb{Q}}(\sigma_{j}))^{125/252}$
\cite[Thm. 2.8]{SunW-2018}. Here we show: 
\begin{thm}
\label{thm:main} Let $(Y,\lambda)$ be a closed, connected oriented
three-manifold with contact form $\lambda$, and let $\Gamma\in H_{1}(Y)$
be such that $c_{1}(\xi)+2\textrm{P.D.}(\Gamma)$ is torsion. Let
$\lbrace\sigma_{j}\rbrace$ be any sequence of nonzero classes in
$ECH(Y,\lambda,\Gamma)$ with definite gradings tending to positive
infinity. Define $d(\sigma_{j})$ by \eqref{eqn:defndk}. Then $d(\sigma_{j})$
is $O(\textrm{gr}^{\mathbb{Q}}(\sigma_{j}))^{2/5}$ as $\textrm{gr}^{\mathbb{Q}}(\sigma_{j})\to+\infty$. 
\end{thm}
We do not know whether or not the $O(\textrm{gr}^{\mathbb{Q}}(\sigma_{j}))^{2/5}$
asymptotics here are optimal \textemdash{} in other words, we do not
know whether there is some contact form on a three-manifold realizing
these asymptotics. We will show in \S\ref{sec:ellipsoid} that there
exist contact forms with $O(1)$ asymptotics for the $d(\sigma_{j})$.
In Remark~\ref{rmk:whycantimprove}, we clarify where the exponent
$\frac{2}{5}$ comes from in our proof, and why the methods in the
current paper can not improve on it.

Another topic which we do not address here, except in a very specific
example, see \S\ref{sec:ellipsoid}, but which is of potential future
interest, is whether the asymptotics of the $d(\sigma_{j})$ carry
interesting geometric information. In this regard, a similar question
in the context of the spectral flow of a one-parameter family of Dirac
operators was recently answered in \cite{Savale-Gutzwiller}. This
is particularly relevant in the context of the argument we give here,
as our argument also involves estimating spectral flow, see Remark~\ref{rmk:whycantimprove}.

\subsection{A dynamical zeta function and a Weyl law}

We now mention two corollaries of Theorem~\ref{thm:main}.

Given $\Gamma\in H_{1}(Y)$, define a set of nonnegative real numbers,
called the \textit{ECH spectrum} for $(Y,\lambda,\Gamma)$ 
\begin{align*}
\Sigma_{(Y,\lambda,\Gamma)} & \coloneqq\cup_{*}\Sigma_{(Y,\lambda,\Gamma),*}\\
\Sigma_{(Y,\lambda,\Gamma),*} & \coloneqq\left\{ c_{\sigma}\left(\lambda\right)|0\neq\sigma\in ECH_{*}(Y,\xi,\Gamma;\mathbb{Z}_{2})\right\} .
\end{align*}
(To emphasize, in the set $\Sigma_{(Y,\lambda,\Gamma),*}$ we are
fixing the grading $*$.) Then, define the Weyl counting function
for $\Sigma_{(Y,\lambda,\Gamma)}$ 
\begin{equation}
N_{(Y,\lambda,\Gamma)}\left(R\right)\coloneqq\#\left\{ c\in\Sigma_{(Y,\lambda,\Gamma)}|c\leq R\right\} .\label{eq:Weyl counting function}
\end{equation}
We now have the following: 
\begin{cor}
\label{cor:Weyl law} If $c_{1}(\xi)+2\textrm{P.D.}(\Gamma)$ is torsion,
then the Weyl counting function \eqref{eq:Weyl counting function}
for the ECH spectrum satisfies the asymptotics 
\begin{equation}
N\left(R\right)=\left[\frac{2^{d}-1}{\textrm{vol}\left(Y,\lambda\right)}\right]R^{2}+O\left(R^{9/5}\right)\label{eq:Weyl law}
\end{equation}
where $d=\textrm{dim }ECH_{\ast}\left(Y,\xi,\Gamma;\mathbb{Z}_{2}\right)+\textrm{dim }ECH_{\ast+1}\left(Y,\xi,\Gamma;\mathbb{Z}_{2}\right)$,
$\ast\gg0$. 
\end{cor}
As another corollary, one may obtain information on the corresponding
dynamical zeta function. To this end, first note that the \textit{ECH
zeta function} 
\begin{equation}
\zeta_{ECH}\left(s;Y,\lambda,\Gamma\right)\coloneqq\sum_{c\ne0\in\Sigma_{(Y,\lambda,\Gamma)}}c^{-s}\label{eq:ECH zeta}
\end{equation}
converges for $\textrm{Re}\left(s\right)>2$ by \eqref{eqn:asymptoticformula}
and defines a holomorphic function of $s$ in this region whenever
$c_{1}(\xi)+2\textrm{P.D.}(\Gamma)$ is torsion, by \eqref{eq:Weyl law}.

In view of for example \cite{Dyatlov-Zworski2016,Giulietti-Liverani-Pollicott},
one can ask if $\zeta_{ECH}$ has a meromorphic continuation to $\mathbb{C}$,
and, if so, whether it contains interesting geometric information.
The Weyl law \eqref{eq:Weyl law} then shows: 
\begin{cor}
\label{cor:zeta function} The zeta function \eqref{eq:ECH zeta}
continues meromorphically to the region $\textrm{Re}\left(s\right)>\frac{5}{3}$.
The only pole in this region is at $s=2$ which is further simple
with residue $\textrm{Res}_{s=2}\zeta_{ECH}\left(s;Y,\lambda,\Gamma\right)=\left[\frac{2^{d}-1}{\textrm{vol}\left(Y,\lambda\right)}\right]$. 
\end{cor}
In \S\ref{sec:ellipsoid}, we give an example of a contact form for
which $\zeta_{ECH}$ has a meromorphic extension to all of $\mathbb{C}$
with two poles at $s=1,2$. The meromorphy and location of the poles
of \eqref{eq:ECH zeta} would be interesting to figure out in general.

\subsection{Idea of the proof and comparison with previous works}

The method of the proof uses previous work by C. Taubes relating embedded
contact homology to Monopole Floer homology. By using Taubes's results,
we can estimate spectral invariants associated to nonzero ECH classes
by estimating the energy of certain solutions of the deformed three-dimensional
Seiberg-Witten equations. This is also the basic idea behind the proofs
of Theorem~\ref{thm:vc} and the result of Sun mentioned above, and
it was inspired by a similar idea in Taubes's proof of the three-dimensional
Weinstein conjecture \cite{Taubes-Weinstein}.

The essential place where our proof differs from these arguments involves
a particular estimate, namely a key ``spectral flow\char`\"{} bound
for families of Dirac operators that appears in all of these proofs.
This estimate bounds the difference between the grading of a Seiberg-Witten
solution, and the ``Chern-Simons\char`\"{} functional, which we review
in \S\ref{subsec:Monopole-Floer-homology}, and is important in all
of the works mentioned above. We prove a stronger bound of this kind
than any previous bound, see Proposition~\ref{prop: irred grading est}
and the discussion about the eta invariant below, and this is the
key point which allows us to prove $O(\textrm{gr}^{\mathbb{Q}}(\sigma_{j}))^{2/5}$
asymptotics. Spectral flow bounds for families of Dirac operators
were also considered in \cite{Savale-Asmptotics,Savale2017-Koszul,Tsai-thesis-paper}.
The main difference here is that in those works the bounds were proved
on reducible solutions where the connections needed to define the
relevant Dirac operators were explicitly given. Here we must consider
irreducible solutions, and so we rely on a priori estimates. 

We have chosen to phrase this spectral flow bound in terms of a bound
on the eta invariants of a family of operators. By the Atiyah-Patodi-Singer
index theorem, the bound we need on the spectral flow is equivalent
to a bound on the eta invariant, and we make the relationship between
these two quantities precise in the appendix.

The paper is organized as follows. In \S\ref{sec:Floer-homologies},
we review what we need to know about embedded contact homology, Monopole
Floer cohomology and Taubes's isomorphism. \S\ref{sec:Estimates-on-Seiberg-Witten}
reviews the eta invariant, reviews the necessary estimates on irreducible
solutions to the Seiberg-Witten equations, and proves the key Proposition~\ref{prop: irred grading est}.
We then give the proof of Theorem~\ref{thm:main} in \S\ref{sec:Asymptotics-of-capacities}
\textemdash{} while our argument in this section is novel, one could
instead argue here as in \cite{SunW-2018}, but we give our own argument
here since it might be of independent interest, see Remark~\ref{rmk:presume}.
The end of the paper reviews the sub-leading asymptotics and the dynamical
zeta function in the case of ellipsoids, and an appendix rephrases
the grading in Seiberg-Witten in terms of the eta invariant rather
than in terms of spectral flow.

\subsection{Acknowledgments}

The first author thanks M. Hutchings, D. McDuff, and W. Sun for very
helpful discussions.

\section{Floer homologies \label{sec:Floer-homologies}}

We begin by reviewing the facts that we will need about ECH and Monopole
Floer homology.

\subsection{\label{subsec:Embedded-contact-homology}Embedded contact homology}

We first summarize what we will need to know about ECH. For more details
and for definitions of the terms that we have not defined here, see
\cite{Hutchings2014lecture}.

Let $(Y,\lambda)$ be a closed oriented three-manifold with a nondegenerate
contact form. Fix a homology class $\Gamma\in H_{1}\left(Y\right)$.
As stated in the introduction, the embedded contact homology $ECH(Y,\lambda,\Gamma)$
is the homology of a chain complex $ECC(Y,\lambda,\Gamma)$. To elaborate,
the chain complex $ECC$ is freely generated over $\mathbb{Z}_{2}$
by \textit{orbit sets} $\alpha=\left\{ \left(\alpha_{j},m_{j}\right)\right\} $
where the $\alpha_{j}$'s are distinct embedded Reeb orbits while
each $m_{j}\in\mathbb{N}$; we further have the constraints that $\sum m_{j}\alpha_{j}=\Gamma\in H_{1}\left(Y\right)$
and $m_{j}=1$ if $\alpha_{j}$ is hyperbolic. To define the chain
complex differential $\partial$, we consider the symplectization
$\left(\mathbb{R}_{t}\times Y,d\left(e^{t}\lambda\right)\right)$,
and choose an almost complex structure $J$ that is $\mathbb{R}$-invariant,
rotates the contact hyperplane $\xi\coloneqq\textrm{ker}\lambda$
positively with respect to $d\lambda$, and satisfies $J\partial_{t}=R$.
The differential on $ECC\left(Y,\lambda,\Gamma\right)$ is now defined
via 
\[
\partial\alpha=\sum_{\beta}\underbrace{\sharp\left[\mathcal{M}_{1}\left(\alpha,\beta\right)/\mathbb{R}\right]}_{\eqqcolon\left\langle \partial\alpha,\beta\right\rangle }\beta.
\]
Here $\mathcal{M}_{1}\left(\alpha,\beta\right)$ denotes the moduli
space of $J$-holomorphic curves $C$ of ECH index $I\left(C\right)=1$
in the symplectization, modulo translation in the $\mathbb{R}$-direction,
and modulo equivalence as currents, with the set of positive ends
given by $\alpha$ and the set of negative ends given by $\beta$.
If $J$ is generic, then the differential squares to zero $\partial^{2}=0$
and defines the ECH group $ECH\left(Y,\lambda;\Gamma\right).$ We
will not review the definition of the ECH index here, see \cite{Hutchings2014lecture}
for more details, but the key point is that the condition $I(C)=1$
forces $C$ to be (mostly) embedded and rigid modulo translation.

As stated in the introduction, the homology $ECH\left(Y,\lambda;\Gamma\right)$
does not depend on the choice of generic $J$, and only depends on
the associated contact structure $\xi$; we therefore denote it $ECH\left(Y,\xi;\Gamma\right)$.
(In fact, the homology only depends on the spin$^{c}$ structure determined
by $\xi$, but we will not need that.) This follows from a canonical
isomorphism between ECH and Monopole Floer homology \cite{Taubes-ECH=HMI},
which we will soon review. The ECH index $I$ induces a relative $\mathbb{Z}/d\mathbb{Z}$
grading on $ECH\left(Y,\xi;\Gamma\right),$ where $d$ is the divisibility
of $c_{1}\left(\xi\right)+2\textrm{P.D.}\left(\Gamma\right)\in H^{2}\left(Y;\mathbb{Z}\right)$
mod torsion. In particular, it is relatively $\mathbb{Z}$-graded
when this second homology class is torsion

Recall now the action of a Reeb orbit from the introduction. This
induces an action on orbit sets $\alpha=\left\{ \left(\alpha_{j},m_{j}\right)\right\} $
by 
\[
\mathcal{A}\left(\alpha\right)\coloneqq\sum_{j=1}^{N}m_{j}\left(\int_{\alpha_{j}}\lambda\right).
\]
The differential decreases action, and so we can define $ECC^{L}\left(Y,\lambda,\Gamma\right)$
to be the homology of the sub-complex generated by orbit sets of action
strictly less than $L$. The homology of this sub-complex $ECH^{L}\left(Y,\lambda,\Gamma\right)$
is again independent of $J$ but now depends on $\lambda$; there
is an inclusion induced map $ECH^{L}\left(Y,\lambda,\Gamma\right)\to ECH\left(Y,\xi,\Gamma\right).$
Using this filtration, we can define the \textit{spectral invariant}
associated to a nonzero class $\sigma$ in ECH 
\[
c_{\sigma}\left(Y,\lambda\right)\coloneqq\inf\left\lbrace L\hspace{1mm}|\hspace{1mm}\sigma\in\textrm{image }(ECH^{L}\left(Y,\lambda,\Gamma\right)\to ECH\left(Y,\xi,\Gamma\right)\right)\rbrace.
\]

As stated in the introduction, the spectral invariants are known to
be $C^{0}$ continuous in the contact form, and so extend to degenerate
contact forms as well by taking a limit over nondegenerate forms,
see \cite{Hutchings2011}.

\subsection{\label{subsec:Monopole-Floer-homology}Monopole Floer homology}

We now briefly review what we need to know about Monopole Floer homology,
referring to \cite{Kronheimer-Mrowka} for additional details and
definitions.

Recall that a\textit{ spin}$^{c}$\textit{ structure} on an oriented
Riemannian three-manifold $Y$ is a pair $\left(S,c\right)$ consisting
of a rank $2$ complex Hermitian vector bundle and a Clifford multiplication
endomorphism $c:T^{*}Y\otimes\mathbb{C}\rightarrow\textrm{End}\left(S\right)$
satisfying $c\left(e_{1}\right)^{2}=-1$ and $c\left(e_{1}\right)c\left(e_{2}\right)c\left(e_{3}\right)=1$
for any oriented orthonormal frame $\left(e_{1},e_{2}e_{3}\right)$
of $T_{y}Y$. Let $\mathfrak{su}\left(S\right)$ denote the bundle
of traceless, skew-adjoint endomorphisms of $S$ with inner product
$\frac{1}{2}\textrm{tr}\left(A^{*}B\right)$. Clifford multiplication
$c$ maps $T^{*}Y$ isometrically onto $\mathfrak{su}\left(S\right)$.
Spin$^{c}$ structures exist on any three-manifold, and the set of
spin$^{c}$ structures is an affine space over $H^{2}(Y;\mathbb{Z})$.
A \textit{spin}$^{c}$ connection $A$ on $S$ is a connection such
that $c$ is parallel. Given two spin-c connections $A_{1},A_{2}$
on $S$, their difference is of the form $A_{1}-A_{2}=a\otimes1^{S}$
for some $a\in\Omega^{1}\left(Y,i\mathbb{R}\right).$ If we denote
by $A_{1}^{t},A_{2}^{t}$ the induced connections on $\det\left(S\right)=\Lambda^{2}S$,
we then have $A_{1}^{t}-A_{2}^{t}=2a$. Hence prescribing a spin$^{c}$
connection on $S$ is the same as prescribing a unitary connection
on $\det\left(S\right)$. We let $\mathcal{A}\left(Y,\mathfrak{s}\right)$
denote the space of all spin$^{c}$ connections on $S$. Given a spin$^{c}$
connection $A$, we denote by $\nabla^{A}$ the associated covariant
derivative. We then define the \textit{spin}$^{c}$ \textit{Dirac
operator} $D_{A}:C^{\infty}\left(S\right)\rightarrow C^{\infty}\left(S\right)$
via $D_{A}\Psi=c\circ\nabla^{A}\Psi$.

Given a spin$^{c}$ structure $\mathfrak{s}=\left(S,c\right)$ on
$Y$,\textit{ monopole Floer homology} assigns three groups denoted
by $\widehat{HM}\left(Y,\mathfrak{s}\right),\widecheck{HM}\left(Y,\mathfrak{s}\right)$
and $\overline{HM}\left(Y,\mathfrak{s}\right)$. These are defined
via infinite dimensional Morse theory on the \textit{configuration
space} $\mathcal{C}\left(Y,\mathfrak{s}\right)=\mathcal{A}\left(Y,\mathfrak{s}\right)\times C^{\infty}\left(S\right)$
using the Chern-Simons-Dirac functional $\mathcal{L}$, defined as
\begin{align}
\mathcal{L}\left(A,\Psi\right)= & \underbrace{-\frac{1}{8}\int_{Y}\left(A^{t}-A_{0}^{t}\right)\wedge\left(F_{A^{t}}+F_{A_{0}^{t}}\right)}_{\eqqcolon CS\left(A\right)}+\frac{1}{2}\int_{Y}\left\langle D_{A}\Psi,\Psi\right\rangle dy\label{eq:CSD}
\end{align}
using a fixed base spin-c connection $A_{0}$ (we pick one with $A_{0}^{t}$
flat in the case of torsion spin-c structures) and a metric $g^{TY}$.

The \textit{gauge group} $\mathcal{G}\left(Y\right)=\textrm{Map}\left(Y,S^{1}\right)$
acts on the configuration space $\mathcal{C}\left(Y,\mathfrak{s}\right)$
by $u.\left(A,\Psi\right)=\left(A-u^{-1}du\otimes I,u\Psi\right).$
The gauge group action is free on the irreducible part $\mathcal{C}^{*}\left(Y,\mathfrak{s}\right)=\left\{ \left(A,\Psi\right)\in\mathcal{C}\left(Y,\mathfrak{s}\right)|\Psi\neq0\right\} \subset\mathcal{C}\left(Y,\mathfrak{s}\right)$
and not free along the reducibles. The blow up of the configuration
space along the reducibles 
\begin{align*}
\mathcal{C}^{\sigma}\left(Y,\mathfrak{s}\right)= & \left\{ \left(A,s,\Phi\right)|\left\Vert \Phi\right\Vert _{L^{2}}=1,s\geq0\right\} 
\end{align*}
then has a free $\mathcal{G}\left(Y\right)$ action $u\cdot\left(A,s,\Phi\right)=\left(A-u^{-1}du\otimes I,s,u\Phi\right).$

To define the Monopole Floer homology groups one needs to perturb
the Chern-Simons-Dirac functional \eqref{eq:CSD}. First given a one
form $\mu\in\Omega^{1}\left(Y;i\mathbb{R}\right)$, one defines the
functional $e_{\mu}\left(A\right)\coloneqq\frac{1}{2}\int_{Y}\mu\wedge F_{A^{t}}$
whose gradient is calculated to be $\ast d\mu$. To achieve non-degeneracy
and transversality of configurations one uses the \textit{perturbed
Chern-Simons-Dirac functional} 
\begin{equation}
\mathcal{L}_{\mu}\left(A,\Psi\right)=\mathcal{L}\left(A,\Psi\right)-e_{\mu}\left(A\right)\label{eq:perturbed CSD}
\end{equation}
where $\mu$ is a suitable finite linear combination of eigenvectors
of $\ast d$ with non-zero eigenvalue. Next let 
\[
\mathbb{T}=\left\{ A\in\mathcal{A}\left(Y,\mathfrak{s}\right)|F_{A^{t}}=0\right\} /\mathcal{G}\left(Y\right)
\]
be the space of $A^{t}$ flat spin-c connections up to gauge equivalence.
We choose a Morse function $f:\mathbb{T}\rightarrow\mathbb{R}$ to
define the functional $\mathfrak{f}:\mathcal{C}^{\sigma}\left(Y,\mathfrak{s}\right)\rightarrow\mathbb{R},\,\mathfrak{f}\left(A_{0}+a,s,\Psi\right)\coloneqq f\left(\left[A_{0}^{t}+a^{h}\right]\right)$,
where $a^{h}$ denotes the harmonic part of $a\in\Omega^{1}\left(Y,i\mathbb{R}\right)$.
The gradient may be calculated $\left(\nabla\mathfrak{f}\right)_{A}^{\sigma}=\left(\left(\nabla f\right)_{\left(A^{t}\right)^{h}},0,0\right)$.

The Monopole Floer homology groups are now defined using solutions
$\left(A,s,\Phi\right)\in\mathcal{C}^{\sigma}\left(Y,\mathfrak{s}\right)$
to the \textit{three-dimensional Seiberg-Witten equations} 
\begin{align}
\frac{1}{2}*F_{A^{t}}+s^{2}c^{-1}\left(\Phi\Phi^{*}\right)_{0}+\left(\nabla f\right)_{p\left(A\right)}+\ast d\mu=0\nonumber \\
s\Lambda\left(A,s,\Phi\right)=0\nonumber \\
D_{A}\Phi-\Lambda\left(A,s,\Phi\right)\Phi=0\label{eq: Seiberg Witten equations blowup}
\end{align}
where $\Lambda\left(A,s,\Phi\right)=\left\langle D_{A}\Phi,\Phi\right\rangle _{L^{2}}$
and $\left(\Phi\Phi^{*}\right)_{0}\coloneqq\Phi\otimes\Phi^{*}-\frac{1}{2}\left|\Phi\right|^{2}$
defines a traceless, Hermitian endormophism of $S.$ We denote by
$\mathfrak{C}$ the set of solutions to the above equations.

We first subdivide the solutions as follows: 
\begin{align*}
\mathfrak{C}^{o}= & \left\{ \left(A,s,\Phi\right)\in\mathfrak{C}|s\neq0\right\} /\mathcal{G}\left(Y\right),\\
\mathfrak{C}^{s}= & \left\{ \left(A,0,\Phi\right)\in\mathfrak{C}|\Lambda\left(A,0,\Phi\right)>0\right\} /\mathcal{G}\left(Y\right)\\
= & \left\{ \left(A,\Phi\right)|\frac{1}{2}F_{A^{t}}+d\mu=0,\:\left[A\right]\textrm{is a critical point of }f,\right.\\
 & \quad\left.\:\Phi\textrm{ is a (positive-)normalized eigenvector of }D_{A}\right\} /\mathcal{G}\left(Y\right)\\
\mathfrak{C}^{u}= & \left\{ \left(A,0,\Phi\right)|\Lambda\left(A,0,\Phi\right)<0\right\} /\mathcal{G}\left(Y\right)\\
= & \left\{ \left(A,\Phi\right)|\frac{1}{2}F_{A^{t}}+d\mu=0,\:\left[A\right]\textrm{is a critical point of }f,\right.\\
 & \quad\left.\:\Phi\textrm{ is a (negative-)normalized eigenvector of }D_{A}\right\} /\mathcal{G}\left(Y\right).
\end{align*}
Next, we consider the free $\mathbb{Z}_{2}$ modules generated by
the three sets above 
\[
C^{o}=\mathbb{Z}_{2}\left[\mathfrak{C}^{o}\right],\:C^{s}=\mathbb{Z}_{2}\left[\mathfrak{C}^{s}\right],\:C^{u}=\mathbb{Z}_{2}\left[\mathfrak{C}^{u}\right].
\]
The chain groups for the three versions of Floer homology mentioned
above are defined by 
\[
\check{C}=C^{o}\oplus C^{s},\,\hat{C}=C^{o}\oplus C^{u},\,\bar{C}=C^{s}\oplus C^{u}.
\]
These chain groups $\check{C},\hat{C},\bar{C}$ can be endowed with
differentials $\check{\partial},\hat{\partial},\bar{\partial}$ with
square zero; we do not give the precise details here, but the idea
is to count Fredholm index one solutions of the four-dimensional equations,
see \cite[Thm. 22.1.4]{Kronheimer-Mrowka} for the details. The homologies
of these three complexes are by definition the three\textit{ monopole
Floer homology groups} 
\[
\widecheck{HM}\left(Y,\mathfrak{s}\right),\widehat{HM}\left(Y,\mathfrak{s}\right),\overline{HM}\left(Y,\mathfrak{s}\right).
\]
They are independent of the choice of metric and perturbations $\mu,\mathfrak{f}$.

Each of the above Floer groups has a relative $\mathbb{Z}/d\mathbb{Z}$
grading where $d$ is the divisibility of $c_{1}\left(S\right)\in H^{2}\left(Y;\mathbb{Z}\right)$
mod torsion. This is defined using the \textit{extended Hessian} 
\begin{align}
\widehat{\mathcal{H}}_{\left(A,\Psi\right)}:C^{\infty}\left(Y;iT^{*}Y\oplus\mathbb{R}\oplus S\right) & \rightarrow C^{\infty}\left(Y;iT^{*}Y\oplus\mathbb{R}\oplus S\right);\;\left(A,\Psi\right)\in\mathcal{C}\left(Y,\mathfrak{s}\right)\nonumber \\
\widehat{\mathcal{H}}_{\left(A,\Psi\right)}\begin{bmatrix}a\\
f\\
\psi
\end{bmatrix} & =\begin{bmatrix}*da+2c^{-1}\left(\psi\Psi\right)_{0}-df\\
-d^{*}a+i\textrm{Re}\left\langle \psi,\Psi\right\rangle \\
c\left(a\right)\Psi+D_{A}\psi+f\Psi
\end{bmatrix}\begin{bmatrix}a\\
f\\
\psi
\end{bmatrix}\nonumber \\
 & =\begin{bmatrix}\ast d & -d & c^{-1}\left(.\Psi\right)_{0}\\
-d^{*} & 0 & \left\langle .,\Psi\right\rangle \\
c\left(.\right)\Psi & .\Psi & D_{A}
\end{bmatrix}\begin{bmatrix}a\\
f\\
\psi
\end{bmatrix}.\label{eq:Hessian}
\end{align}
The relative grading between two irreducible generators $\mathfrak{a}_{i}=\left(A_{i},s_{i},\Phi_{i}\right)$,
($s_{i}\neq0$), $i=1,2$, is now defined via $\textrm{gr}\left(\mathfrak{a}_{1},\mathfrak{a}_{2}\right)=\textrm{sf}\left\{ \widehat{\mathcal{H}}_{\left(A_{t},\Psi_{t}\right)}\right\} _{0\leq t\leq1}$
(mod $d$) for some path of configurations $\left(A_{t},\Psi_{t}\right)$
starting at $\left(A_{2},s_{2}\Phi_{2}\right)$ and ending at $\left(A_{1},s_{1}\Phi_{1}\right)$,
where $\textrm{sf}$ denotes the spectral flow.

In the case when the spin-c structure is torsion, the monopole Floer
groups are further equipped with an \textit{absolute} $\mathbb{Q}$-\textit{grading},
refining this relative grading. As we will review in the appendix,
this is given via 
\begin{equation}
\textrm{gr}^{\mathbb{Q}}\left[\mathfrak{a}\right]=\begin{cases}
2k-\eta\left(D_{A}\right)+\frac{1}{4}\eta_{Y}-\frac{1}{2\pi^{2}}CS\left(A\right); & \mathfrak{a}=\left(A,0,\Phi_{k}^{A}\right)\in\mathfrak{C}^{s},\\
-\eta\left(\widehat{\mathcal{H}}_{\left(A,s\Phi'\right)}\right)+\frac{5}{4}\eta_{Y}-\frac{1}{2\pi^{2}}CS\left(A\right); & \mathfrak{a}=\left(A,s,\Phi\right)\in\mathfrak{C}^{o},\,s\neq0.
\end{cases}\label{eq:absolute grading}
\end{equation}
where $\Phi_{k}^{A}$ above denotes the $k$th positive eigenvector
of $D_{A}$ (see \S\ref{sec:Absolute--grading}), and $\eta_{Y}$
and $\eta_{D_{A}}$ denote the \textit{eta invariant} of the corresponding
operator, which we will review in \S\ref{sec:Estimates-on-Seiberg-Witten}.

\subsection{ECH=HM\label{subsec:ECH=00003D00003D00003D00003D00003DHM}}

We now state the isomorphism between the ECH and HM, proved in \cite{Taubes-ECH=HMI}.
Given a contact manifold $\left(Y^{3},\lambda\right)$ with $d\lambda$-compatible
almost complex structure $J$ as before, we define a metric $g^{TY}$
via $g^{TY}|_{\xi}=d\lambda\left(.,J.\right)$, $\left|R\right|=1$
and $R$ and $\xi$ are orthogonal. This metric is adapted to the
contact form in the sense $*d\lambda=2\lambda,\left|\lambda\right|=1.$
Decompose $\xi\otimes\mathbb{C}=\underbrace{K}_{\xi^{1,0}}\oplus\underbrace{K^{-1}}_{\xi^{0,1}}$
into the $i,-i$ eigenspaces of $J$. The contact structure now determines
the canonical spin-c structure $\mathfrak{s}^{\xi}$ via $S^{\xi}=\mathbb{C}\oplus K^{-1}$
with Clifford multiplication $c^{\xi}$ given by 
\begin{align*}
c^{\xi}\left(R\right) & =\begin{bmatrix}i\\
 & -i
\end{bmatrix},\\
c^{\xi}\left(v\right) & =\begin{bmatrix} & -i_{v^{1,0}}\\
v^{0,1}\wedge
\end{bmatrix},\quad v\in\xi.
\end{align*}
Furthermore, there is a unique spin-c connection $A_{c}$ on $S^{\xi}$
with the property that $D_{A_{c}}\begin{bmatrix}1\\
0
\end{bmatrix}=0$ and we call the induced connection $A_{c}^{t}$ on $K^{-1}=\textrm{det}\left(S^{\xi}\right)$
the canonical connection. Tensor product with an auxiliary Hermitian
line bundle $E$ via $S^{E}=S^{\xi}\otimes E$ and $c^{E}=c^{\xi}\otimes1$
gives all other spin-c structures $\mathfrak{s}^{E}$. Furthermore
all spin-c connections on $S^{E}$ arise as $A=A_{c}\otimes1+1\otimes\nabla^{A}$
for some unitary connection $\nabla^{A}$ on $E$. The ECH/HM isomorphism
is then 
\begin{equation}
\widecheck{HM}_{*}\left(-Y,\mathfrak{s}^{E}\right)=ECH_{*}\left(Y,\xi;\textrm{ P.D.}c_{1}\left(E\right)\right).\label{eq:HM=00003D00003D00003D00003D00003DECH}
\end{equation}
In the literature, this isomorphism is often stated with the left
hand side given by the cohomology group $\widehat{HM}^{*}(Y)$ instead;
the point is that $\widehat{HM}^{*}(Y)$ and $\widecheck{HM}(-Y)$
are canonically isomorphic, see \cite[S 22.5, Prop. 28.3.4]{Kronheimer-Mrowka}.
The isomorphism \eqref{eq:HM=00003D00003D00003D00003D00003DECH} allows
us to define a $\mathbb{Q}$-grading on ECH, by declaring that \eqref{eq:HM=00003D00003D00003D00003D00003DECH}
preserves this $\mathbb{Q}$-grading.

We now state the main ideas involved in the isomorphism (we restrict
attention to the case when $c_{1}\left(\det\mathfrak{s}^{E}\right)$
is torsion, which is the case which is relevant here, and we sometimes
state estimates that, while true, are stronger than those originally
proved by Taubes). To this end, let $\sigma\in\widecheck{HM}\left(-Y,\mathfrak{s}^{E}\right)$.
We use the perturbed Chern-Simons-Dirac functional \eqref{eq:perturbed CSD}
and its gradient flow \eqref{eq: Seiberg Witten equations blowup}
with $\mu=ir\lambda$, $r\in\left[0,\infty\right)$, in defining monopole
Floer homology. (One also adds a small term $\eta$ to $\mu$ to achieve
transversality, see for example \cite{Cristofaro-Gardiner-Hutchings-Gripp2015},
but to simplify the notation we will for now suppress this term.)
Giving a family of (isomorphic) monopole Floer groups $\widecheck{HM}\left(-Y,\mathfrak{s}^{E}\right)$,
the class $\sigma$ is hence representable by a formal sum of solutions
to \eqref{eq: Seiberg Witten equations blowup} corresponding to $\mu=ir\lambda$.
Denote by $\check{C^{r}}$ the $\mu=ir\lambda$ version of the complex
$\check{C}$ and note that its reducible generators are all of the
form $\mathfrak{a}=\left(A,0,\Phi_{k}\right)$ where $A=A_{0}-ir\lambda$,
$A_{0}^{t}$ is flat and $\Phi_{k}$ is the $k$th positive eigenvector
of $D_{A}$. An important estimate $\eta\left(D_{A_{0}-ir\lambda}\right)=O\left(r\right)$
now gives that the grading of this generator $\textrm{gr}^{\mathbb{Q}}\left[\mathfrak{a}\right]=\frac{r^{2}}{4\pi^{2}}\int\lambda\wedge d\lambda+O(r)>\textrm{gr}^{\mathbb{Q}}\left[\sigma\right]$
by \eqref{eq:absolute grading} for $r\gg0$. Hence for $r\gg0$ the
class $\sigma$ is represented by a formal sum of irreducible solutions
to \eqref{eq: Seiberg Witten equations blowup} with $\mu=ir\lambda$,
and by a max-min argument, we may choose a family $\left(A_{r},\Psi_{r}\right)\coloneqq\left(A_{r},s_{r}\Phi_{r}\right)$
satisfying 
\[
\textrm{gr}^{\mathbb{Q}}\left[\sigma\right]=\textrm{gr}^{\mathbb{Q}}\left[\left(A_{r},\Psi_{r}\right)\right].
\]
Following a priori estimates on solutions to the Seiberg-Witten equations,
one then proves another important estimate $\eta\left(\widehat{\mathcal{H}}_{\left(A_{r},\Psi_{r}\right)}\right)=O\left(r^{3/2}\right)$
uniformly in the class $\sigma$. This gives $CS\left(A_{r}\right)=O\left(r^{3/2}\right)$
which in turn by a differential relation (see \S\ref{sec:Asymptotics-of-capacities})
leads to $e_{\lambda}\left(A_{r}\right)=O\left(1\right)$. The final
step in the proof shows that for any sequence of solutions $\left(A_{r},\Psi_{r}\right)$
to Seiberg-Witten equations with $e_{\lambda}\left(A_{r}\right)$
bounded, the $E$-component $\Psi_{r}^{+}$ of the spinor $\Psi_{r}=\begin{bmatrix}\Psi_{r}^{+}\\
\Psi_{r}^{-}
\end{bmatrix}\in C^{\infty}\left(Y;\underbrace{E\oplus K^{-1}E}_{=S^{E}}\right)$ satisfies the weak convergence $\left(\Psi_{r}^{+}\right)^{-1}\left(0\right)\rightharpoonup\left\{ \left(\alpha_{j},m_{j}\right)\right\} $
to some ECH orbit set. This last orbit set is what corresponds to
the image of $\sigma\in\widecheck{HM}\left(Y,\mathfrak{s}^{E}\right)$
in ECH under the isomorphism \eqref{eq:HM=00003D00003D00003D00003D00003DECH}.
Furthermore, crucially for our purposes, one has 
\begin{equation}
c_{\sigma}\left(\lambda\right)=\lim_{r\rightarrow\infty}\frac{e_{\lambda}\left(A_{r}\right)}{2\pi},\label{eq:capacity limit of energy}
\end{equation}
see \cite[Prop. 2.6]{Cristofaro-Gardiner-Hutchings-Gripp2015}. (The
proof in \cite[Prop. 2.6]{Cristofaro-Gardiner-Hutchings-Gripp2015}
is given in the case where $\lambda$ is nondegenerate, but it holds
for all $\lambda$ by continuity.)

\section{Estimating the eta invariant}

\label{sec:Estimates-on-Seiberg-Witten}

Let $D$ be a generalized Dirac operator acting on sections of a Clifford
bundle $E$ over a closed, oriented Riemannian manifold $Y$. Then
the sum 
\begin{equation}
\eta(D,s)\coloneqq\sum_{\lambda\neq0}\frac{\textrm{sgn}(\lambda)}{|\lambda|^{s}}\label{eqn:sumdefinition}
\end{equation}
is a convergent analytic function of a complex variable $s$, as long
as $\textrm{Re}(s)$ is sufficiently large; here, the sum is over
the nonzero eigenvalues of $D$. Moreover, the function $\eta(D,s)$
has an analytic continuation to a meromorphic function on $\mathbb{C}$
of $s$, which we also denote by $\eta(D,s)$, and which is holomorphic
near $0$. We now define 
\[
\eta(D)\coloneqq\eta(D,0).
\]
We should think of this as a formal signature of $D$, which we call
the \textit{eta invariant} of Atiyah-Patodi-Singer \cite{APSI}.

We will be primarily concerned with the case where $D=D_{A_{r}}$,
namely $D$ is the spin-c Dirac operator for a connection $A_{r}$
solving \eqref{eq: Seiberg Witten equations blowup}. Another case
of interest to us is where $D$ is the \textit{odd signature operator}
on $C^{\infty}\left(Y;iT^{*}Y\oplus\mathbb{R}\right)$ sending 
\[
\left(a,f\right)\mapsto\left(\ast da-df,d^{*}a\right),
\]
in which case we denote the corresponding $\eta$ invariant by $\eta_{Y}$.

Now consider the Seiberg-Witten equations \eqref{eq: Seiberg Witten equations blowup}
corresponding to $\mu=ir\lambda$, for a torsion spin$^{c}$ structure
as above, and note that an irreducible solution (after rescaling the
spinor) corresponds to a solution $\left(A_{r},\Psi_{r}\right)$ to
the Seiberg-Witten equations on $\mathcal{C}\left(Y,\mathfrak{s}\right)$
given via 
\begin{align}
\frac{1}{2}c\left(*F_{A^{t}}\right)+r\left(\Psi\Psi^{*}\right)_{0}+c\left(ir\lambda\right) & =0\nonumber \\
D_{A}\Psi & =0.,\label{eq: irred SW}
\end{align}
A further small perturbation is needed to obtain transversality of
solutions see \cite[S 2.1]{Cristofaro-Gardiner-Hutchings-Gripp2015}.
We ignore these perturbation as they make no difference to the overall
argument.

We can now state the primary result of this section: 
\begin{prop}
\label{prop: irred grading est}Any solution to \eqref{eq: irred SW}
satisfies $\eta\left(\widehat{\mathcal{H}}_{\left(A_{r},\Psi_{r}\right)}\right)=O\left(r^{\frac{3}{2}}\right)$. 
\end{prop}
The purpose of the rest of the section will be to prove this.

\subsection{Known estimates}

We first collect some known estimates on solutions to the equations
\eqref{eq: irred SW}. 
\begin{lem}
For some constants $c_{q}$, $q=0,1,2,\ldots$, we have 
\begin{equation}
\left|\nabla^{q}F_{A^{t}}\right|\leq c_{q}\left(1+r^{1+q/2}\right).\label{eqn:keyequation}
\end{equation}
\end{lem}
\begin{proof}
We first note that we have the estimates: 
\begin{align*}
\left|\Psi_{r}^{+}\right| & \leq1+\frac{c_{0}}{r}\\
\left|\Psi_{r}^{-}\right| & \leq\frac{c_{0}}{r}\left(\left|1-\left|\Psi_{r}^{+}\right|^{2}\right|+\frac{1}{r}\right)\\
\left|\left(\nabla^{A}\right)^{q}\Psi_{r}^{+}\right| & \leq c_{q}\left(1+r^{q/2}\right)\\
\left|\left(\nabla^{A}\right)^{q}\Psi_{r}^{-}\right| & \leq c_{q}\left(1+r^{\left(q-1\right)/2}\right)
\end{align*}
The first two of these estimates are proved in \cite[Lem. 2.2]{Taubes-Weinstein}.
The third and fourth are proved in \cite[Lem. 2.3]{Taubes-Weinstein}.

The lemma now follows by combining the above estimates with the equation
\eqref{eq: irred SW}. 
\end{proof}
In \eqref{sec:Asymptotics-of-capacities}, we will also need: 
\begin{lem}
One has the bound 
\begin{equation}
\left|CS\left(A_{r}\right)\right|\leq c_{0}r^{2/3}e_{\lambda}\left(A_{r}\right)^{4/3}\label{eq: CS 4/3 est}
\end{equation}
where the constant $c_{0}$ only depends on the metric contact manifold. 
\end{lem}
\begin{proof}
This is proven in \cite[eq. 4.9]{Taubes-Weinstein}, see also \cite[Lem. 2.7]{Cristofaro-Gardiner-Hutchings-Gripp2015}. 
\end{proof}

\subsection{\label{subsec:Grading-estimates}The $\eta$ invariant of families
of Dirac operators}

In this section, we prove the key Proposition~\ref{prop: irred grading est}.
The main point that we need is the following fact concerning the $\eta$
invariant: 
\begin{prop}
\label{prop:keyprop} Let $A_{r}$ be a solution to \eqref{eq: irred SW}.
Then $\eta\left(D_{A_{r}}\right)$ is $O(r^{3/2})$ as $r\to\infty$. 
\end{prop}
Before giving the proof, we first explain our strategy.

The first point is that we have the following integral formula for
the $\eta$ invariant: 
\begin{equation}
\eta(D_{A_{r}})=\frac{1}{\sqrt{\pi t}}\int_{0}^{\infty}\textrm{tr }(D_{A_{r}}e^{-tD_{A_{r}}^{2}})dt\label{eqn:integralformula}
\end{equation}
where the right hand side is a convergent integral. This is proved
in \cite[S 2]{Bismut-Freed-II}, by Mellin transform it is equivalent
to the fact that the eta function $\eta(D_{A_{r}},s)$ in \eqref{eqn:sumdefinition}
is holomorphic for $Re(s)>-2.$

We therefore have to estimate the integral in \eqref{eqn:integralformula}.
To do this, we will need the following estimates: 
\begin{lem}
\label{lem:keyestimates} There exists a constant $c_{0}$ independent
of $r$ such that for all $r\ge1,t>0$: 
\begin{align*}
\left|\textrm{tr}\left(D_{A_{r}}e^{-tD_{A_{r}}^{2}}\right)\right| & \leq c_{0}r^{2}e^{c_{0}rt},\quad\textrm{ and}\\
\left|\textrm{tr}\left(e^{-tD_{A_{r}}^{2}}\right)\right| & \leq c_{0}t^{-3/2}e^{c_{0}rt}.
\end{align*}
\end{lem}
Once we have proved Lemma~\ref{lem:keyestimates}, Proposition~\ref{prop:keyprop}
will follow from a short calculation, which we will give at the end
of this section.

The proof of Lemma~\ref{lem:keyestimates} will require two auxiliary
lemmas, see Lemma~\ref{lem:summary} and Lemma~\ref{lem:technical}
below, and some facts about the heat equation associated to a Dirac
operator that we will now first recall. Let $D$ be a Dirac operator
on a Clifford bundle $V$ over a closed manifold $Y$. The \textit{heat
equation} associated to $D$ is the equation 
\[
\frac{\partial s}{\partial t}+D^{2}s=0
\]
for sections $s$, and nonnegative time $t$; the operator $e^{-tD^{2}}$
is the solution operator for this equation. The heat equation has
an associated \textit{heat kernel} $H_{t}(x,y)$ which is a (time-dependent)
section of the bundle $V\boxtimes V$ over $Y\times Y$ whose fiber
over a point $(x,y)$ is $V_{x}\otimes V_{y}^{*}$; it is smooth for
$t>0$. For any smooth section $s$ of $V$ and $t>0$, the heat kernel
satisfies 
\begin{equation}
e^{-tD^{2}}s(x)=\int_{Y}H_{t}(x,y)s(y)\textrm{vol}(y).\label{eqn:kernelproperty}
\end{equation}
Also, 
\begin{equation}
\left[\frac{\partial}{\partial t}+D_{x}^{2}\right]H_{t}(x,y)=0,\label{eqn:solvedirac}
\end{equation}
where $D_{x}$ denotes the Dirac operator applied in the $x$ variables.

Moreover, 
\begin{equation}
\textrm{tr}(e^{-tD^{2}})=\int_{Y}\textrm{tr}(H_{t}(y,y))\textrm{vol}(y).\label{eqn:traceformula}
\end{equation}
Hence, we can prove Lemma~\ref{lem:keyestimates} by bounding $|H_{t}|$
along the diagonal. The operator $De^{-tD^{2}}$ has a kernel $L_{t}(x,y)$
as well, and the analogous results hold.

A final fact we will need is {\textit{Duhamel's principle}}: this
says that the inhomogeneous heat equation 
\[
\frac{\partial\tilde{s}}{\partial t}+D^{2}\tilde{s}_{t}=s_{t}
\]
has a unique solution tending to $0$ with $t$, given by 
\begin{equation}
\tilde{s}_{t}(x)=\int_{0}^{t}(e^{-(t-t')D^{2}}s_{t'})(x)dt',\label{eqn:duhamel}
\end{equation}
as long as $s_{t}$ is a smooth section of $S$, continuous in $t$.

Now let $D$ be $D_{A_{r}}$, and $V$ the spinor bundle for the spin$^{c}$
structure $S$, and let $H_{t}^{r}$ and $L_{t}^{r}$ be defined as
above, but with $D=D_{A_{r}}$. Let $\rho(x,y)$ the Riemannian distance
function. Define an auxiliary function 
\[
h_{t}(x,y)\coloneqq\left(4\pi t\right)^{-3/2}e^{-\frac{\rho\left(x,y\right)^{2}}{4t}}.
\]

In the case of $Y=\mathbb{R}^{3}$, with $\rho$ the standard Euclidean
distance, the function $h_{t}(x,y)$ is precisely the ordinary heat
kernel. In our case, the kernel $H_{t}^{r}(x,y)$ has an \textit{asymptotic
expansion} as $t\to0$, 
\begin{equation}
H_{t}^{r}(x,y)\sim h_{t}\left(x,y\right)(b_{0}^{r}(x,y)+b_{1}^{r}(x,y)t+b_{2}^{r}(x,y)t^{2}+\ldots+),\label{eqn:asymptoticexpansion}
\end{equation}
that is studied in detail in \cite[Ch. 2]{Berline-Getzler-Vergne};
here, the $b_{i}^{r}(x,y)$ are defined on all of $Y\times Y$. The
following lemma summarizes what we need to know about the results
from \cite[Ch. 2]{Berline-Getzler-Vergne}: 
\begin{lem}
\label{lem:summary}There exists for all $i=0,1,2,\ldots$ sections
$b_{i}^{r}(x,y)$ such that: 
\begin{itemize}
\item The $b_{i}^{r}$ are supported in any neighborhood of the diagonal. 
\item The asymptotic expansion \eqref{eqn:asymptoticexpansion} may be formally
differentiated to obtain asymptotic expansions for the derivative.
In particular, there is an asymptotic expansion 
\begin{equation}
L_{t}^{r}(x,y)\sim h_{t}\left(x,y\right)(\tilde{b}_{0}^{r}(x,y)+\tilde{b}_{1}^{r}(x,y)t+\tilde{b}_{2}^{r}(x,y)t^{2}+\ldots+)\label{eqn:lasymptotic}
\end{equation}
where 
\[
\tilde{b}_{n}^{r}(x,y)=(D_{A_{r}}+c(\rho d\rho/2t))b_{n}^{r}(x,y).
\]
\item For any $n,t>0$, 
\[
L_{t}^{r}(x,y)-h_{t}(x,y)\sum_{i=0}^{n}\tilde{b}_{i}^{r}(x,y)t^{i}
\]
is $O(t^{i-1/2})$, in the $C^{0}$- norm on the product, as $t\to0$. 
\item 
\begin{equation}
(\partial_{t}+D_{A_{r}}^{2})\left(L_{t}^{r}(x,y)-h_{t}(x,y)\sum_{i=0}^{n}\tilde{b}_{i}^{r}(x,y)t^{i}\right)=-D_{A_{r}}^{2}\tilde{b}_{n}^{r}t^{n}.\label{eqn:nonhom}
\end{equation}
\end{itemize}
\end{lem}
\begin{proof}
The lemma summarizes those parts of the proof of \cite[Thm. 2.30]{Berline-Getzler-Vergne}
that we will soon need; the arguments in \cite[Thm. 2.30]{Berline-Getzler-Vergne}
provide the proof. The idea behind the first bullet point is that
$h_{t}(x,y)$ is on the order of $t^{\infty}$ away from the diagonal.
The reason for the $i-1/2$ exponent in the third bullet point is
that $h_{t}$ has a $t^{-3/2}$ term. For the fourth bullet point,
the point is that the coefficients $b_{i}^{r}$ are constructed so
as to satisfy \eqref{eqn:solvedirac} when formally differentiating
\eqref{eqn:asymptoticexpansion} and equating powers of $t$; this
gives a recursion which is relevant for our purposes because it implies
that when we truncate the expansion at a finite $n$, the inhomogeneous
equation \eqref{eqn:nonhom} is satisfied. 
\end{proof}
In view of the first bullet point of the above lemma, we only have
to understand the coefficients $b_{i}^{r}$ in a neighborhood of the
diagonal. To facilitate this, let $i_{g^{TY}}$ denote the injectivity
radius of the Riemannian metric $g$, and given $y\in Y$, let $B_{y}\left(\frac{i_{g^{TY}}}{2}\right)$
denote a geodesic ball of radius $\frac{i_{g^{TY}}}{2}$ centered
at $y$, and let $y$ denote a choice of co-ordinates on this ball.
Define $G_{y}^{k}\subset C^{\infty}\left(B_{y}\left(\frac{i_{g^{TY}}}{2}\right)\right)$
to be the subspace of ($r$-dependent) functions $f$ satisfying the
estimate $\partial_{y}^{\alpha}f=O\left(r^{k+\frac{\left|\alpha\right|}{2}}\right)$
as $r\to\infty$, $\forall\alpha\in\mathbb{N}_{0}^{3},$ and for each
$j\in\frac{1}{2}\mathbb{N}_{0}$, further define the subspace $W_{y}^{j}\subset C^{\infty}\left(B_{y}\left(\frac{i_{g^{TY}}}{2}\right)\right)$
via 
\begin{align}
f & \in W_{y}^{j}\iff f=\sum_{i=1}^{N}f_{i},\quad\textrm{ with each }f_{i}\in y^{\alpha}G_{y}^{k},\,k\leq j+\frac{\left|\alpha\right|}{2}.\label{eq: space for heat coeffs.}
\end{align}
Finally, given $y\in Y$, we choose a convenient frame for $S_{y}^{\xi}$
and $E$ over $B_{y}\left(\frac{i_{g^{TY}}}{2}\right)$, which we
will call a \textit{synchronous frame}; specifically, choose an orthonormal
basis for each of $S_{y}^{\xi},E_{y}$, and parallel transport along
geodesics with $A_{c},A_{r}$ to obtain local orthonormal trivializations
$\left\{ s_{1},s_{2}\right\} ,\left\{ e\right\} $. Now, if $b$ is
any section of $S\otimes S_{y}$ over $B_{y}(\frac{i_{g^{TY}}}{2})\times\lbrace y\rbrace$
write 
\begin{equation}
b(\cdot,y)=\sum_{k,l=1}^{2}f_{b,kl}^{y}\left(.\right)\left(s_{k}\otimes e\right)\left(.\right)\left(s_{l}\otimes e\right)^{*}\left(y\right).\label{eq:coefficient function heat ker}
\end{equation}

\begin{lem}
\label{lem:technical} There is a constant $c_{0}$ independent of
$r$ such that for any $t>0$, $r\geq1$, we have 
\begin{equation}
|H_{t}^{r}(x,y)|\leq c_{0}h_{2t}\left(x,y\right)e^{c_{0}rt}.\label{eq: basic heat kernel estimate}
\end{equation}
Further, for any $y\in Y$, the restriction of the terms $b_{j}^{r}$
to $B_{y}\left(\frac{i_{g^{TY}}}{2}\right)\times\lbrace y\rbrace$
have the property that their corresponding functions $f_{b_{j}^{r},kl}^{y}$
in \eqref{eq:coefficient function heat ker} are all in $W_{y}^{j}$. 
\end{lem}
\begin{proof}
The first bullet point is similar to \cite[Prop. 3.1]{Savale-Asmptotics}.

To prove the second bullet point, we use the fact that the terms $b_{j}^{r}$
in the heat kernel expression \eqref{eqn:asymptoticexpansion} are
known to satisfy a recursion, as alluded to above, and explained in
the proof of \cite[Thm. 2.30]{Berline-Getzler-Vergne}. Specifically,
fix $y\in Y$, choose geodesic coordinates $y$ around $y$, mapping
$0$ to $y$, and choose a synchronous frame as in \eqref{eq:coefficient function heat ker}.
Then, in these coordinates, we have 
\begin{equation}
b_{0}^{r}(x,y)=\sum_{i=1}^{2}g^{-1/4}(x)(s_{i}\otimes e)(x)(s_{i}\otimes e)^{*}(y),\label{eqn:base}
\end{equation}
where $g=\det\left(g_{jk}\right)$. Moreover, if use these coordinates
to identify sections of $S\otimes S_{y}^{*}$ with a vector of functions,
then we have 
\begin{equation}
b_{j}^{r}(x,y)=-\frac{1}{g^{1/4}(x)}\int_{0}^{1}\rho^{j-1}g^{1/4}\left(\rho x\right)D_{A_{r}}^{2}b_{j-1}^{r}\left(\rho x,y\right)d\rho,\quad j\geq1\label{eqn:recursion}
\end{equation}
where on the right hand side of this equation, we mean that we are
integrating this vector of functions component by component.

Now recall the Bochner-Lichnerowicz-Weitzenbock formula for the Dirac
operator: 
\begin{equation}
D_{A_{r}}^{2}=\nabla_{A_{r}}^{*}\nabla_{A_{r}}+\frac{\kappa}{4}-\frac{1}{2}c(*F_{A_{r}}),\label{eqn:weitzenboch}
\end{equation}
where $\kappa$ denotes the scalar curvature; we will want to combine
this with \eqref{eqn:recursion}. In coordinates, we have 
\begin{equation}
\nabla_{A_{r}}=(\partial_{1}+\Gamma_{1},\partial_{2}+\Gamma_{2},\partial_{3}+\Gamma_{3}),\label{eqn:nabla}
\end{equation}
where each $\Gamma_{i}$ is the $i^{th}$ Christoffel symbol for $A_{r}$.
We also have 
\begin{equation}
\nabla_{A_{r}}^{*}=-\sum_{j,k}^{3}-g^{jk}(\partial_{k}+\Gamma_{k})+\sum_{i,j,k}^{3}g^{jk}\Gamma_{jk}^{i},\label{eqn:nabla*}
\end{equation}
where the $\Gamma_{j,k}^{i}$ are the Christoffel symbols of the Riemannian
metric. Since we have $A_{r}=1\otimes A+A_{c}\otimes1$, where $A_{c}$
is the canonical connection on $S^{\xi}$, we can decompose each Christoffel
symbol 
\begin{equation}
\Gamma_{i}=c_{i}+a_{i},\label{eqn:decompose}
\end{equation}
where the $c_{i}$ are Christoffel symbols for $A_{c}$ and the $a_{i}$
are Christoffel symbols for $A$.

The $c_{i}$ are independent of $r$. To understand the $a_{j}$,
first write the defining equations for the curvature 
\[
F_{kj}=\partial_{k}a_{j}-\partial_{j}a_{k}.
\]
Now write the coordinate $x=(x^{1},x^{2},x^{3})$, and consider $\sum_{k=1}^{3}x^{k}(\partial_{k}a_{j}-\partial_{j}a_{k}).$
Reintroducing the radial coordinate $\rho$, we have 
\[
\rho\sum_{k=1}^{3}x^{k}\partial_{k}a_{j}=\frac{\partial a_{j}}{\partial\rho}.
\]
On the other hand, since the frame $e$ is parallel, we have $\nabla_{x^{1}\partial_{x_{1}}+x^{2}\partial_{x_{2}}+x^{3}\partial_{x_{3}}}e=\nabla_{\rho\partial_{\rho}}e=0$,
hence $\sum_{k=1}^{3}x^{k}a_{k}=0$. Thus, we have 
\begin{equation}
a_{j}\left(x\right)=\sum_{k=1}^{3}\int_{0}^{1}d\rho\rho x^{k}F_{kj}\left(\rho x\right).\label{eq: connection coefficient}
\end{equation}
In particular, it follows from the a priori estimate \eqref{eqn:keyequation}
and \eqref{eq: connection coefficient} that each 
\begin{equation}
a_{j}\in W_{y}^{\frac{1}{2}}.\label{eqn:keyequation}
\end{equation}

Now note that we have $W_{y}^{j}+W_{y}^{k}\subset W_{y}^{\max\left\{ j,k\right\} }$,
$W_{y}^{j}\cdot W_{y}^{k}\subset W_{y}^{j+k}$, and $\partial_{y}W_{y}^{j}\subset W_{y}^{j+\frac{1}{2}}$.
Hence, by \eqref{eqn:weitzenboch}, \eqref{eqn:nabla}, \eqref{eqn:nabla*},
and \eqref{eqn:keyequation}, we have that the square of the Dirac
operator has the schematic form 
\begin{equation}
D_{A_{r}}^{2}=\sum_{j,k}-g^{jk}\partial_{j}\partial_{k}+P_{j}\partial_{j}+Q\label{eq: Dirac form}
\end{equation}
where $P_{j}\in W^{\frac{1}{2}}$ and $Q\in W_{y}^{1}$. The Lemma
now follows by induction, using \eqref{eqn:base} and \eqref{eqn:recursion}. 
\end{proof}
We now give the promised: 
\begin{proof}[Proof of Lemma~\ref{lem:keyestimates}]

The second bullet point follows by combining \eqref{eqn:traceformula}
and \eqref{eq: basic heat kernel estimate}.

To prove the first bullet point, our strategy will be to bound the
pointwise size of the kernel $L_{t}^{r}(y,y)$ and appeal to the version
of \eqref{eqn:traceformula} for $L_{t}^{r}$.

To do this, consider the asymptotic expansion \eqref{eqn:lasymptotic}.
By a theorem of Bismut-Freed (\cite[Thm. 2.4]{Bismut-Freed-II}),
for any $y\in Y$, $\textrm{tr}L_{t}^{r}(y,y)$ is $O\left(t^{1/2}\right)$
as $t\to0$. So we have $\textrm{tr}L_{t}^{r}(y,y)=\textrm{tr }R_{t}^{r}\left(y,y\right)$
for the remainder 
\begin{align*}
R_{t}^{r}(x,y)\coloneqq & L_{t}^{r}-D_{A_{r}}\left[h_{t}\left(b_{0}+tb_{1}\right)\right].
\end{align*}
By \eqref{eqn:nonhom}, $R_{t}^{r}$ satisfies the inhomogeneous heat
equation 
\[
\left(\partial_{t}+D_{A_{r}}^{2}\right)R_{t}^{r}(x,y)=h_{t}t\left\{ -D_{A_{r}}^{3}b_{1}+c\left(\frac{\rho d\rho}{2t}\right)D_{A_{r}}^{2}b_{1}\right\} ,
\]
and by the third bullet point of Lemma~\ref{lem:summary}, $R_{t}^{r}\to0$
as $t\to0$. We can then apply Duhamel's principle \eqref{eqn:duhamel}
to write 
\begin{equation}
R_{t}^{r}(x,y)=\int_{0}^{t}e^{-(t-s)D_{A_{r}}^{2}}h_{s}(x,y)s\underbrace{\left\{ -D_{A}^{3}b_{1}+c\left(\frac{\rho d\rho}{2s}\right)D_{A}^{2}b_{1}\right\} }_{\eqqcolon K_{s}}ds.\label{eq: duhamel application}
\end{equation}
We can then apply the key property of the heat kernel \eqref{eqn:kernelproperty}
to write 
\[
R_{t}^{r}(x,y)=\int_{0}^{t}\int_{Y}H_{t-s}^{r}(x,z)h_{s}(z,y)sK_{s}(z,y)\textrm{vol}(z)ds,
\]
and we can apply the second bullet point of Lemma~\ref{lem:summary}
to conclude that 
\[
|R_{t}^{r}(y,y)|\le c_{0}\int_{0}^{t}\int_{Y}e^{c_{0}r(t-s)}h_{s}(z,y)h_{2(t-s)}(z,y)sK_{s}(z,y)\textrm{vol}(z)ds.
\]
By the first bullet point of Lemma~\ref{lem:summary}, we can assume
that $K_{s}(z,y)$ is supported in $B_{y}\left(\frac{i_{g^{TY}}}{2}\right)\times\lbrace y\rbrace$.
Thus, we just have to bound 
\begin{equation}
\int_{0}^{t}\int_{B_{y}\left(\frac{i_{g^{TY}}}{2}\right)}e^{c_{0}r(t-s)}h_{s}(y,0)h_{2(t-s)}(y,0)sK_{s}(y,0)dyds,\label{eqn:thingtobound}
\end{equation}
where $y$ are geodesic coordinates centered at $y$. To do this,
choose a synchronous frame for the spinor bundle, as we have been
doing above. Then, following \eqref{eqn:decompose}, \eqref{eq: connection coefficient},
in these coordinates the Dirac operator is seen to have the form 
\[
D_{A_{r}}=w^{jk}\partial_{j}+K,
\]
for $r-$independent $w^{jk}$ and $K\in W_{y}^{\frac{1}{2}}$, in
the geodesic coordinates and orthonormal frame introduced before.
Combining this with the second bullet point of Lemma~\ref{lem:technical}
gives that the term $K_{s}\in W_{y}^{\frac{5}{2}}$. So, \eqref{eqn:thingtobound}
is dominated by a finite sum of integrals of the form 
\[
\int_{0}^{t}ds\int_{B_{y}\left(\frac{i_{g^{TY}}}{2}\right)}dy\,se^{c_{0}rt}h_{2(t-s)}\left(y,0\right)h_{s}\left(y,0\right)y^{\alpha}r^{k},\quad k\leq\frac{5}{2}+\frac{\left|\alpha\right|}{2}.
\]

On $B_{y}\left(\frac{i_{g^{TY}}}{2}\right),$ we have 
\[
y^{I}h_{t}\left(y,0\right)\leq c_{1}t^{\frac{1}{2}|I|}h_{2t}\left(y,0\right),
\]
for some constant $c_{1}$. Hence, we can bound the above integral
by 
\begin{equation}
c_{1}r^{k}e^{c_{0}rt}\int_{0}^{t}ds\int_{B_{y}\left(\frac{i_{g^{TY}}}{2}\right)}dy\,s^{1+\frac{|\alpha|}{2}}h_{2(t-s)}\left(y,0\right)h_{2s}\left(y,0\right)dy.\label{eqn:tobound2}
\end{equation}
We also have 
\[
\int_{Y}h_{t}(x,y)h_{t'}(y,z)dy\le c_{2}h_{4(t+t')}(x,z)
\]
as proved in \cite[Sec. A]{Savale-Asmptotics}. So, we can bound \eqref{eqn:tobound2}
by 
\[
c_{3}r^{k}e^{c_{0}rt}\int_{0}^{t}s^{1+\frac{|\alpha|}{2}}h_{8t}(0,0)ds=c_{3}r^{k}e^{c_{0}rt}\int_{0}^{t}s^{1+\frac{|\alpha|}{2}}(4t)^{-3/2}ds\le c_{3}r^{k}t^{\frac{1+|\alpha|}{2}}e^{c_{0}rt}.
\]
We know that $k-\frac{1+|\alpha|}{2}\le2$. So, putting all of the
above together, \eqref{eqn:thingtobound} is dominated by a finite
sum of terms of the form 
\[
r^{2}r^{\frac{1+|\alpha|}{2}}t^{\frac{1+|\alpha|}{2}}e^{c_{0}rt}
\]
which proves the result. 
\end{proof}
We can finally give: 
\begin{proof}[Proof of Proposition~\ref{prop:keyprop}]
Define $E(x)\coloneqq\text{sign}(x)\text{erfc}(|x|)=\text{sign}(x)\cdot\frac{2}{\sqrt{\pi}}\int_{|x|}^{\infty}e^{-s^{2}}ds<e^{-x^{2}}$.
This is a rapidly decaying function, so the function $E(D_{A_{r}})$
is defined, and its trace is a convergent sum 
\[
\textrm{tr}(E(D_{A_{r}}))=\sum_{\lambda}E(\lambda),
\]
where $\lambda$ is an eigenvalue of $D_{A_{r}}.$ The eta invariant
in unchanged under positive rescaling 
\[
\eta\left(D_{A_{r}}\right)=\eta\left(\frac{1}{\sqrt{r}}D_{A_{r}}\right).
\]
Now use \eqref{eqn:integralformula} to rewrite the right hand side
of the above equation as 
\[
\left|\int_{0}^{1}dt\frac{1}{\sqrt{\pi t}}\textrm{ tr}\left[\frac{1}{\sqrt{r}}D_{A_{r}}e^{-\frac{t}{r}D_{A_{r}}^{2}}\right]+\textrm{tr}\,E\left(\frac{1}{\sqrt{r}}D_{A_{r}}\right)\right|.
\]
The absolute value of the first summand in the above expression is
bounded from above by a constant multiple of $r^{3/2}$, by the first
bullet point in Lemma~\ref{lem:keyestimates}. The absolute value
of the second summand in the same expression is bounded from above
by $\textrm{tr}\,e^{-\frac{1}{r}D_{A_{r}}^{2}}$, which by the second
bullet point in Lemma~\ref{lem:keyestimates} is bounded by a constant
multiple of $r^{3/2}$ as well. 
\end{proof}
\begin{proof}[Proof of Proposition~\ref{prop: irred grading est}]
An application of the Atiyah-Patodi-Singer index theorem as in \S\ref{sec:Absolute--grading}
gives 
\[
\frac{1}{2}\eta\left(\widehat{\mathcal{H}}_{\left(A_{r},\Psi_{r}\right)}\right)=\frac{1}{2}\eta\left(D_{A_{r}}\right)+\frac{1}{2}\eta_{Y}+\textrm{sf}\left\{ \widehat{\mathcal{H}}_{\left(A_{r},\epsilon\Psi_{r}\right)}\right\} _{0\leq\varepsilon\leq1}.
\]
The spectral flow term above is estimated to be $O\left(r^{3/2}\right)$
as in \cite[S 5.4]{Taubes-Weinstein} while $\eta\left(D_{A_{r}}\right)=O\left(r^{3/2}\right)$
by Proposition~\ref{prop:keyprop}. 
\end{proof}
We also note that the constant in Proposition~\ref{prop: irred grading est}
above is only a function of $\left(Y,\lambda,J\right)$ and independent
of the class $\sigma=\left[\left(A_{r},\Psi_{r}\right)\right]\in\widecheck{HM}\left(-Y,\mathfrak{s}^{E}\right)$
defined by the Seiberg-Witten solution. 
\begin{rem}
\label{rmk:whycantimprove}

The reason that we can not improve upon $\textrm{gr}^{\mathbb{Q}}$
asymptotics is because we do not know how to strengthen the $O\left(r^{3/2}\right)$
spectral flow estimate on the irreducible solutions of Propositions
~\ref{prop: irred grading est} or \ref{prop:keyprop}. A better
$O(r)$ estimate does however exist \cite{Savale-Asmptotics,Savale-Gutzwiller}
for reducible solutions for which one understands the connection precisely
in the limit $r\rightarrow\infty$. However, the a priori estimates
\eqref{eqn:keyequation} are not strong enough to carry out the same
for irreducibles. 
\end{rem}

\section{\label{sec:Asymptotics-of-capacities}Asymptotics of capacities}

\subsection{The main theorem}

In this section we now prove our main theorem Theorem~\ref{thm:main}
on ECH capacities. 
\begin{proof}[Proof of Theorem~\ref{thm:main}]

Let $0\neq\sigma_{j}\in ECH\left(Y,\lambda,\Gamma\right)$, $j=0,1,2,\ldots$,
be a sequence of non-vanishing classes with definite gradings $\textrm{gr}^{\mathbb{Q}}(\sigma_{j})$
tending to positive infinity. As in \S\ref{subsec:ECH=00003D00003D00003D00003D00003DHM},
we use the perturbed Chern-Simons-Dirac functional \eqref{eq:perturbed CSD}
$\mathcal{L}_{\mu}$ and its gradient flow \eqref{eq: Seiberg Witten equations blowup}
with $\mu=ir\lambda$, $r\in\left[0,\infty\right)$, in defining monopole
Floer homology. Hence for each $r\in\left[1,\infty\right)$, the class
$\sigma_{j}$ may be represented by a formal sum of solutions to \eqref{eq: Seiberg Witten equations blowup}
with $\mu=ir\lambda$. As noted in \S\ref{subsec:ECH=00003D00003D00003D00003D00003DHM},
this solution is eventually irreducible. Without loss of generality
we may assume 
\[
\textrm{gr}^{\mathbb{Q}}(\sigma_{j})=q+j,
\]
where $q$ is a fixed rational number and $j\in2\mathbb{N}$.

We now estimate $r_{1}\left(j\right)$, the infimum of the values
of $r$ such that each solution $\left[\mathfrak{a}_{j}\right]_{r}$
to \eqref{eq: Seiberg Witten equations blowup} representing $\sigma_{j}$
is irreducible. For this note that a reducible solution is of the
form $\mathfrak{a}=\left(A,0,\Phi_{k}\right)$ where $A=A_{0}-ir\lambda$,
$A_{0}^{t}$ flat and $\Phi_{k}$ the $k$th positive eigenvector
of $D_{A}$. The grading of such a reducible is given by \eqref{eq:absolute grading}.
The important estimate $\left|\eta_{D_{A_{1}+r\lambda}}\right|\leq c_{0}r$,
(\cite[Thm. 1.2]{Savale2017-Koszul}) now shows $\textrm{gr}^{\mathbb{Q}}\left[\left(A,0,\Phi_{k}\right)\right]>\textrm{gr}^{\mathbb{Q}}\left(\sigma_{j}\right)=q+j$
for 
\[
r>\bar{r}_{1}\left(j\right)\coloneqq\sup\left\{ r|\frac{r^{2}}{4\pi^{2}}\textrm{vol}\left(Y,\lambda\right)<c_{0}r+q+j\right\} .
\]
Hence $r_{1}\left(j\right)<\bar{r}_{1}\left(j\right)$. Furthermore
\begin{equation}
\bar{r}_{1}\left(j\right)=2\pi\left[\frac{j}{\textrm{vol}\left(Y,\lambda\right)}\right]^{1/2}+O\left(1\right)\textrm{ as }j\rightarrow\infty\label{eq:asymptotics of r1}
\end{equation}
from the above definition. A max-min argument, as also mentioned in
\S\ref{subsec:ECH=00003D00003D00003D00003D00003DHM} then gives $\forall j\in2\mathbb{N}$
a piecewise-smooth family of irreducible solutions $\left[\mathfrak{a}\right]_{r}=\left(A_{r},\Psi_{r}\right)$
, $r>r_{1}\left(j\right)$, of fixed grading $\textrm{gr}^{\mathbb{Q}}\left[\mathfrak{a}\right]=q+j$
such that $\mathcal{L}_{\mu}$ is continuous, see (for example) \cite[S 2.6]{Cristofaro-Gardiner-Hutchings-Gripp2015}. 

By \eqref{eq: CS 4/3 est}, we have 
\begin{equation}
\left|CS\left(A\right)\right|\leq c_{0}r^{2/3}e_{\lambda}\left(A\right)^{4/3}.\label{eq: CS vs E}
\end{equation}
In addition, by combining \eqref{eq:absolute grading} and Proposition~\ref{prop:keyprop},
we have 
\begin{equation}
\left|\frac{1}{2\pi^{2}}CS\left(A_{r}\right)-(q+j)\right|\leq c_{0}r^{3/2},\label{eq: irred. spectral flow}
\end{equation}
with the constant $c_{0}>0$ being independent of the grading $j$.
We also have the differential relation 
\[
r\frac{de_{\lambda}}{dr}=\frac{dCS}{dr}
\]
between the two functionals, away from the discrete set of points
where derivatives are undefined, see \cite[Lem. 2.5]{Cristofaro-Gardiner-Hutchings-Gripp2015}.
Now define $F\left(r\right)=\frac{1}{2}r_{1}^{2}\textrm{vol}\left(Y,\lambda\right)+\int_{r_{1}}^{r}e_{\lambda}\left(A_{s}\right)ds$.
This is a continuous function, and $v$ is continuous as well, so
we may integrate the above equation to conclude that 
\begin{align*}
CS\left(r\right) & =rF'-F
\end{align*}
valid for all $r$ away from the above discrete set; here, we have
used \cite[Property 2.3.(i)]{SunW-2018}, together with the computation
in \cite[Lem. 2.3]{Cristofaro-Gardiner-Hutchings-Gripp2015} in the
computation of the terms at $r_{1}$.


On account of \eqref{eq: irred. spectral flow}, $F$ is then a super/subsolution
to the ODEs 
\[
-c_{2}r^{3/2}\leq rF'-F-\left(q+j\right)\leq c_{2}r^{3/2}
\]
for $r\geq r_{1}$. This gives 
\begin{align}
\frac{1}{2}r_{1}^{2}\textrm{vol}\left(Y,\lambda\right)+r\left[\frac{q+j}{r_{1}}-\frac{q+j}{r}-2c_{2}r^{1/2}+2c_{2}r_{1}^{1/2}\right]\leq & F\nonumber \\
 & \parallel\nonumber \\
\frac{1}{2}r_{1}^{2}\textrm{vol}\left(Y,\lambda\right)+r\left[\frac{q+j}{r_{1}}-\frac{q+j}{r}+2c_{2}r^{1/2}-2c_{2}r_{1}^{1/2}\right]\geq & F\nonumber \\
\frac{1}{2r}r_{1}^{2}\textrm{vol}\left(Y,\lambda\right)+\frac{q+j}{r_{1}}-3c_{2}r^{1/2}+2c_{2}r_{1}^{1/2}\leq & F'\nonumber \\
 & \parallel\nonumber \\
\frac{1}{2r}r_{1}^{2}\textrm{vol}\left(Y,\lambda\right)+\frac{q+j}{r_{1}}+3c_{2}r^{1/2}-2c_{2}r_{1}^{1/2}\geq & F'.\label{eq: subsolution estimates}
\end{align}

Next the estimate \eqref{eq: CS vs E} in terms of $F$ is 
\begin{equation}
-c_{1}r^{2/3}\left(F'\right)^{4/3}\leq rF'-F\leq c_{1}r^{2/3}\left(F'\right)^{4/3}.\label{eq:diff ineq F'}
\end{equation}
We let $\rho_{0}$ be the smallest positive root of $\frac{1}{3}-\left[\rho+\rho^{2}+\rho^{3}+\rho^{4}\right]=0$
and define 
\begin{align*}
\bar{r}_{2}\left(j\right) & =\sup\left\{ r|c_{1}r^{-2/3}F^{1/3}\geqslant\rho_{0}\right\} 
\end{align*}
which is finite on account of \eqref{eq: subsolution estimates}.
Further with $c_{3}=1+3\left(\frac{2c_{1}}{3}\right)+3\left(\frac{2c_{1}}{3}\right)^{2}$
define 
\begin{align}
\tilde{r}_{2}\left(j\right) & =\sup\left\{ r|\frac{1}{2r}r_{1}^{2}\textrm{vol}\left(Y,\lambda\right)+\frac{q+j}{r_{1}}+3c_{2}r^{1/2}-2c_{2}r_{1}^{1/2}\geqslant\left(\frac{3}{4c_{1}}\right)^{3}r\right.\nonumber \\
 & \qquad\quad\textrm{or}\quad\frac{1}{2r}r_{1}^{2}\textrm{vol}\left(Y,\lambda\right)+\frac{q+j}{r_{1}}-\frac{q+j}{r}+2c_{2}r^{1/2}-2c_{2}r_{1}^{1/2}\geq\left(\frac{1}{9c_{3}}\right)^{3}r\nonumber \\
 & \qquad\quad\textrm{or}\quad\left.\frac{1}{2r}r_{1}^{2}\textrm{vol}\left(Y,\lambda\right)+\frac{q+j}{r_{1}}-\frac{q+j}{r}+2c_{2}r^{1/2}-2c_{2}r_{1}^{1/2}\geq r\right\} \label{eq: def r2 ttilde}
\end{align}
and set $r_{2}\left(j\right)\coloneqq\max\left\{ \bar{r}_{2}\left(j\right),\tilde{r}_{2}\left(j\right)\right\} $.
We note that $r_{2}\left(j\right)=O\left(j^{1/2}\right)$. We now
have the following lemma. 
\begin{lem}
For $r>r_{2}\left(j\right)$ we have 
\begin{equation}
\left(\frac{F}{r}\right)^{1/3}-\frac{2c_{1}}{3r}F^{2/3}\leq\left(F'\right)^{1/3}\leq\left(\frac{F}{r}\right)^{1/3}+\frac{2c_{1}}{3r}F^{2/3}.\label{eq: diff ineq. 2}
\end{equation}
\end{lem}
\begin{proof}
By definition, 
\begin{align*}
 & r>r_{2}\left(j\right)\geq\bar{r}_{2}\left(j\right)\\
\implies & \rho\coloneqq c_{1}r^{-2/3}F^{1/3}<\rho_{0}\\
\implies & \rho+\rho^{2}+\rho^{3}+\rho^{4}<\frac{1}{3}
\end{align*}
as well as 
\begin{align}
 & r>r_{2}\left(j\right)\geq\tilde{r}_{2}\left(j\right)\nonumber \\
\implies & F'\leq\frac{1}{2r}r_{1}^{2}\textrm{vol}\left(Y,\lambda\right)+\frac{q+j}{r_{1}}+3c_{2}r^{1/2}-2c_{2}r_{1}^{1/2}<\left(\frac{3}{4c_{1}}\right)^{3}r\label{eq: easy inequality}
\end{align}
by \eqref{eq: subsolution estimates}.

For $y=\left(F'\right)^{1/3}$ equations \eqref{eq:diff ineq F'}
become the pair of quartic inequalities 
\begin{align}
0 & \leq c_{1}r^{-1/3}y^{4}-y^{3}+r^{-1}F\label{eq:quartic1}\\
0 & \leq c_{1}r^{-1/3}y^{4}+y^{3}-r^{-1}F\label{eq:quartic2}
\end{align}
With $y_{0}^{\pm}=\left(\frac{F}{r}\right)^{1/3}\pm\frac{2c_{1}}{3r}F^{2/3}$
we calculate 
\begin{align*}
 & c_{1}r^{-1/3}\left(y_{0}^{+}\right)^{4}-\left(y_{0}^{+}\right)^{3}+r^{-1}F\\
= & -r^{-5/3}c_{1}F^{4/3}\left[1-\frac{4}{3}\rho-\frac{64}{27}\rho^{2}-\frac{32}{27}\rho^{3}-\frac{16}{81}\rho^{4}\right]<0\quad\textrm{and }\\
 & c_{1}r^{-1/3}\left(y_{0}^{-}\right)^{4}+\left(y_{0}^{-}\right)^{3}-r^{-1}F\\
= & r^{-5/3}c_{1}F^{4/3}\left[-1-\frac{4}{3}\rho+\frac{64}{27}\rho^{2}-\frac{32}{27}\rho^{3}+\frac{16}{81}\rho^{4}\right]<0.
\end{align*}
Since the minimum of the quartic \eqref{eq:quartic1} is attained
at $y_{\textrm{min}}=\left(F'\right)^{1/3}=\frac{3}{4c_{1}}r^{1/3}$,
this gives $\left(F'\right)^{1/3}=y\leq y_{0}^{+}$ or $\left(F'\right)^{1/3}=y\geq y_{\textrm{min}}=\frac{3}{4c_{1}}r^{1/3}$.
The second possibility being disallowed on account of \eqref{eq: easy inequality},
gives the desired upper bound of \eqref{eq: diff ineq. 2}. Similarly,
the minimum of quartic \eqref{eq:quartic2} is attained at the negative
$y_{\textrm{min}}=-\frac{3}{4c_{1}}r^{1/3}$. Hence $y_{0}^{-}\leq y=\left(F'\right)^{1/3}$
which is the lower bound in \eqref{eq: diff ineq. 2}. 
\end{proof}
Next we cube \eqref{eq: diff ineq. 2} and use \eqref{eq: def r2 ttilde}
to obtain 
\begin{equation}
\frac{F}{r}-c_{3}\frac{F^{4/3}}{r^{5/3}}\leq F'\leq\frac{F}{r}+c_{3}\frac{F^{4/3}}{r^{5/3}}\label{eq:diff ineq. F 3}
\end{equation}
for $r\geq r_{2}\left(j\right)$. This gives 
\begin{align}
r^{1/3}\left[\frac{r^{1/3}}{-c_{3}+c_{0}^{-}r^{1/3}}\right]\leq F^{1/3} & \leq r^{1/3}\left[\frac{r^{1/3}}{c_{3}+c_{0}^{+}r^{1/3}}\right],\quad\textrm{ where }\nonumber \\
c_{0}^{\pm} & =\frac{r_{3}^{1/3}\mp c_{3}\left[\frac{F\left(r_{3}\right)}{r_{3}}\right]^{1/3}}{F\left(r_{3}\right)^{1/3}}\label{eq: final estimate F}
\end{align}
for $r\geq r_{3}\left(j\right)\geq r_{2}\left(j\right)$. The last
equation with \eqref{eq:diff ineq. F 3} gives 
\begin{align*}
\frac{F\left(r_{3}\right)}{\left(r_{3}^{1/3}+c_{3}\left[\frac{F\left(r_{3}\right)}{r_{3}}\right]^{1/3}\right)^{3}}=\frac{1}{\left(c_{0}^{-}\right)^{3}}\leq & \underbrace{F'\left(r\right)}_{=e_{\lambda}\left(A_{r}\right)}\leq\frac{1}{\left(c_{0}^{+}\right)^{3}}=\frac{F\left(r_{3}\right)}{\left(r_{3}^{1/3}-c_{3}\left[\frac{F\left(r_{3}\right)}{r_{3}}\right]^{1/3}\right)^{3}}
\end{align*}
and hence 
\[
\frac{F\left(r_{3}\right)}{r_{3}}\left[1-\frac{4c_{3}}{r_{3}^{1/3}}\left(\frac{F\left(r_{3}\right)}{r_{3}}\right)^{1/3}\right]\leq e_{\lambda}\left(A_{r}\right)\leq\frac{F\left(r_{3}\right)}{r_{3}}\left[1+\frac{4c_{3}}{r_{3}^{1/3}}\left(\frac{F\left(r_{3}\right)}{r_{3}}\right)^{1/3}\right]
\]
from \eqref{eq: def r2 ttilde} for $r\gg0$. However, \eqref{eq: subsolution estimates}
for $r=r_{3}$ when substituted into the last equation above becomes
\begin{align*}
\left[\frac{q+j}{r_{1}}-\frac{q+j}{r_{3}}-2c_{2}r_{3}^{1/2}\right]\left(1-R\right) & \leq e_{\lambda}\left(A_{r}\right)\leq\left[\frac{1}{2r_{3}}r_{1}^{2}\textrm{vol}\left(Y,\lambda\right)+\frac{q+j}{r_{1}}+2c_{2}r_{3}^{1/2}\right]\left(1+R\right)\\
\textrm{with }\; & R\coloneqq\frac{4c_{3}}{r_{3}^{1/3}}\left[\frac{1}{2r_{3}}r_{1}^{2}\textrm{vol}\left(Y,\lambda\right)+\frac{q+j}{r_{1}}+2c_{2}r_{3}^{1/2}\right]^{1/3}.
\end{align*}
Setting $r_{3}=j^{4/5}$ (satisfying $r_{3}\geq r_{2}=O\left(j^{1/2}\right)$
for $j\gg0$) and using \eqref{eq:capacity limit of energy}, \eqref{eq:asymptotics of r1}
gives 
\begin{equation}
c_{\sigma_{j}}\left(\lambda\right)=j^{\frac{1}{2}}\textrm{vol}\left(Y,\lambda\right)^{\frac{1}{2}}+O\left(j^{2/5}\right),\label{eq:final estimate}
\end{equation}
as $j\rightarrow\infty$, which is our main result Theorem~\ref{thm:main}. 
\end{proof}
\begin{rem}
\label{rmk:presume} One could replace the arguments in this subsection
with the arguments in Sun's paper \cite{SunW-2018}, if desired \textemdash{}
the key reason why we have a stronger bound than Sun is because of
our stronger bound on the Chern-Simons functional, and not because
of anything we do in this subsection. We have chosen to include our
argument here, which we developed independently of the arguments in
\cite{SunW-2018}, for completeness, and because it might be of independent
interest, although we emphasize that we do use the result of Sun establishing
\cite[Property 2.3.(i)]{SunW-2018}.

On the other hand, the arguments in \cite{Cristofaro-Gardiner-Hutchings-Gripp2015}
are not quite strong enough for Theorem~\ref{thm:main}, even with
the improved bound in Proposition~\ref{prop:keyprop}. 
\end{rem}

\subsection{Proofs of Corollaries}

Here we prove the two corollaries Corollary~\ref{cor:Weyl law} and
Corollary~\ref{cor:zeta function}, both following immediately from
the capacity formula Theorem~\ref{thm:main}. 
\begin{proof}[Proof of Corollary~\ref{cor:Weyl law}]
The $\mathbb{Z}_{2}$ vector space $ECH(Y,\xi,\Gamma;\mathbb{Z}_{2})$
is known to be two-periodic, and nontrivial, in sufficiently high
grading, see for example \cite{Hutchings2014lecture}. Thus, for $\ast\gg0$
sufficiently large there exists a finite set of classes $\left\{ \sigma_{1},\ldots,\sigma_{2^{d}-1}\right\} \subset ECH_{\ast}\left(Y,\xi,\Gamma;\mathbb{Z}_{2}\right)\cup ECH_{\ast+1}\left(Y,\xi,\Gamma;\mathbb{Z}_{2}\right)$
such that 
\[
\left\{ 0,U^{j}\sigma_{1},\ldots,U^{j}\sigma_{2^{d}-1}\right\} =ECH_{\ast+2j}\left(Y,\xi,\Gamma;\mathbb{Z}_{2}\right)\cup ECH_{\ast+1+2j}\left(Y,\xi,\Gamma;\mathbb{Z}_{2}\right),\;\forall j\geq0.
\]
Thus the ECH spectrum modulo a finite set is given by 
\[
\cup_{j=0}^{\infty}\left\{ c_{U^{j}\sigma_{1}}\left(\lambda\right),\ldots,c_{U^{j}\sigma_{2^{d}-1}}\left(\lambda\right)\right\} .
\]
The corollary now follows as $c_{U^{j}\sigma_{l}}\left(\lambda\right)=j^{\frac{1}{2}}\textrm{vol}\left(Y,\lambda\right)^{\frac{1}{2}}+O\left(j^{2/5}\right)$,
$1\leq l\leq2^{d}-1$, by Theorem~\ref{thm:main}. 
\end{proof}
\begin{proof}[Proof of Corollary~\ref{cor:zeta function}]
As in the previous corollary, the ECH zeta function is given, modulo
a finite and holomorphic in $s\in\mathbb{C}$, sum by 
\[
\sum_{j=0}^{\infty}\left[c_{U^{j}\sigma_{1}}\left(\lambda\right)^{-s}+\ldots+c_{U^{j}\sigma_{2^{d}-1}}\left(\lambda\right)^{-s}\right].
\]
With $\zeta^{R}\left(s\right)$ denoting the Riemann zeta function,
we may using Theorem~\ref{thm:main} compare 
\[
\left|\sum_{j=0}^{\infty}c_{U^{j}\sigma_{1}}\left(\lambda\right)^{-s}-\textrm{vol}\left(Y,\lambda\right)^{-\frac{s}{2}}\underbrace{\sum_{j=0}^{\infty}j^{-\frac{s}{2}}}_{=\zeta^{R}\left(\frac{s}{2}\right)}\right|=O\left(\sum_{j=0}^{\infty}\frac{1}{j^{3s/5}}\right)
\]
whence the difference is holomorphic for $\textrm{Re}\left(s\right)>\frac{5}{3}$.
The corollary now follows on knowing $s=2$ to be the only pole of
the Riemann zeta function $\zeta^{R}\left(\frac{s}{2}\right)$ with
residue $1$. 
\end{proof}

\subsection{The ellipsoid example}

\label{sec:ellipsoid}

We close by presenting an example with $O(1)$ asymptotics, and where
the corresponding $\zeta_{ECH}$ function extends meromorphically
to all of $\mathbb{C}$.

Consider the \textit{symplectic ellipsoid} 
\[
E(a,b)\coloneqq\left\lbrace \frac{|z_{1}|^{2}}{a}+\frac{|z_{2}|^{2}}{b}\le1\right\rbrace \subset\mathbb{C}^{2}=\mathbb{R}^{4}.
\]
The symplectic form on $\mathbb{R}^{4}$ has a standard primitive
\[
\lambda_{std}=\frac{1}{2}\sum_{i=1}^{2}(x_{i}dy_{i}-y_{i}dx_{i}).
\]
This restricts to $\partial E(a,b)$ as a contact form, and the ECH
spectrum of $(\partial E(a,b),\lambda)$ is known. Specifically, let
$N(a,b)$ be the sequence whose $j^{th}$ element (indexed starting
at $j=0$) is the $(j+1)^{st}$ smallest element in the matrix 
\[
(ma+nb)_{(m,n)\in\mathbb{Z}_{\ge0}\times\mathbb{Z}_{\ge0}}.
\]
Then, the ECH spectrum $S_{\partial E(a,b)}$ is precisely the values
in $N(a,b)$. Moreover, the homology 
\[
ECH_{*}(\partial E(a,b))
\]
has a canonical $\mathbb{Z}$-grading, such that the empty set of
Reeb orbits has grading $0$, and it is known to have one generator
$\sigma_{j}$ in each grading $2j$, see \cite{Hutchings2014lecture}.
The spectral invariant associated to $\sigma_{j}$ is precisely the
$j^{th}$ element in the sequence $N(a,b)$.

With this understood, we now have: 
\begin{prop}
Let $\sigma_{j}$ be any sequence of classes in $ECH(\partial E(a,b))$
with grading tending to infinity. Then, the $d(\sigma_{j})$ are $O(1)$.
In fact, if $a/b$ is irrational, then 
\[
\lim_{j\to\infty}\frac{d(\sigma_{j})}{j}=\frac{a+b}{2}.
\]
\end{prop}
\begin{proof}
Assume that $a/b$ is irrational. If $t=c_{\sigma_{j}}$, then by
the above description, the grading of $\sigma_{j}$ is precisely twice
the number of terms in $N(a,b)$ that have value less than $t$. With
this understood, the example follows from \cite[Lem. 2.1]{Cristofaro-Gardiner-Xueshan-Li-Stanley2018}.

When $a/b$ is rational, a similar argument still works to show $O(1)$
asymptotics. Namely, if $t=c_{\sigma_{j}}$, then by above, the grading
of $\sigma_{j}$ is precisely twice the number of terms in $N(a,b)$
that have value less than $t$, up to an error no larger than some
constant multiple of $\sqrt{j}$. Now apply \cite[Thm. 2.10]{Beck-Robins2015}. 
\end{proof}
\begin{prop}
The ECH $\zeta$ function $\zeta_{ECH}$ for $ECH(\partial E(a,b))$
has a meromorphic continuation to all of $\mathbb{C}$. It has exactly
two poles, at $s=1$ and $s=2$, with residues 
\[
\textrm{Res}_{s=2}\zeta_{ECH}\left(s;Y,\lambda,\Gamma\right)=\frac{1}{ab},\quad\textrm{Res}_{s=1}\zeta_{ECH}\left(s;Y,\lambda,\Gamma\right)=\frac{1}{2}\left(\frac{1}{a}+\frac{1}{b}\right).
\]
\end{prop}
\begin{proof}
The ECH zeta function in this example 
\begin{align}
\zeta_{ECH}\left(s;Y,\lambda,\Gamma\right) & =\sum_{m,n\in\mathbb{N}}\left(ma+nb\right)^{-s}\nonumber \\
 & =\frac{1}{2}\left[\zeta^{B}\left(s,a|a,b\right)+\zeta^{B}\left(s,b|a,b\right)-a^{-s}\zeta^{R}\left(s\right)-b^{-s}\zeta^{R}\left(s\right)\right]\label{eq:relation zeta functions}
\end{align}
is given in terms of the classical zeta functions of Riemann and Barnes
\cite{Barnes1901,Ruijsenaars2000} 
\[
\zeta^{B}\left(s,w|a,b\right)\coloneqq\sum_{m,n\in\mathbb{N}_{0}}\left(w+ma+nb\right)^{-s},\;w\in\mathbb{R}_{>0}.
\]
Thus $\zeta_{ECH}\left(s;Y,\lambda,\Gamma\right)$ \eqref{eq:relation zeta functions}
is known to possess a meromorphic continuation to the entire complex
plane in this example. Its only two poles are at $s=1,2$ with residues
\begin{align*}
\textrm{Res}_{s=2}\zeta_{ECH}\left(s;Y,\lambda,\Gamma\right) & =\frac{1}{ab}\\
\textrm{Res}_{s=1}\zeta_{ECH}\left(s;Y,\lambda,\Gamma\right) & =\frac{1}{2}\left(\frac{1}{a}+\frac{1}{b}\right)
\end{align*}
respectively; while its values at the non-positive integers are also
known \cite[Cor. 2.4]{Spreafico2009}. In particular its value at
zero is 
\[
\zeta_{ECH}\left(0;Y,\lambda,\Gamma\right)=\frac{1}{4}+\frac{1}{12}\left(\frac{b}{a}+\frac{a}{b}\right).
\]
\end{proof}

\appendix

\section{\label{sec:Absolute--grading}The $\mathbb{Q}$-grading and the $\eta$
invariant}

In this appendix we give a formula for the absolute grading $\textrm{gr}^{\mathbb{Q}}$
on monopole Floer groups of torsion spin-c structures (from \cite[S 28.3]{Kronheimer-Mrowka})
in terms of a relevant eta invariant. First to recall the definition
of $\textrm{gr}^{\mathbb{Q}}\left(\mathfrak{a}\right),\;\mathfrak{a}=\left(A,s,\Phi\right)$,
choose a four manifold $X$ which bounds $Y$ and form the manifold
with cylindrical end $Z=X\cup\left(Y\times\left[0,\infty\right)_{t}\right)$.
Choose a metric $g^{TZ}$ on $Z$ which is of product type $g^{TZ}=g^{TX}+dt^{2}$
on the cylindrical end. Choose a spin-c structure $\left(S^{TZ},c^{Z}\right)$
over $Z$ which is of the form 
\begin{eqnarray*}
S^{TZ}=S_{+}^{TZ}\oplus S_{-}^{TZ}; &  & S_{+}^{TZ}=S_{-}^{TZ}=S^{TY},\\
c^{Z}\left(\alpha\right) & = & \begin{bmatrix}0 & c^{Y}\left(\alpha\right)\\
c^{Y}\left(\alpha\right) & 0
\end{bmatrix},\quad\alpha\in T^{*}Y\\
c^{Z}\left(dt\right) & = & \begin{bmatrix}0 & -I\\
I & 0
\end{bmatrix}
\end{eqnarray*}
over the cylindrical end and a spin-c connection on $S^{TZ}$ of the
form $B=dt\wedge\partial_{t}+A\oplus A$ on the cylindrical end. One
may now form the Fredholm elliptic operator 
\[
d^{*}+d^{+}+D_{B}^{+}:L_{1}^{2}\left(Z;\Lambda_{TZ}^{1}\oplus S_{+}^{TZ}\right)\rightarrow L^{2}\left(Z;\Lambda_{TZ}^{0}\oplus\Lambda_{TZ}^{2,+}\oplus S_{-}^{TZ}\right).
\]
The absolute grading of $\left[\left(A,0,\Phi_{0}^{A}\right)\right]\in\mathfrak{C}^{s},$
with $A^{t}$ flat and $\Phi_{0}^{A}$ the first positive eigenvector
of $D_{A},$ is given in terms of this operator. The precise formula
\cite[Defn. 28.3.1]{Kronheimer-Mrowka} simplifies to 
\[
\textrm{gr}^{\mathbb{Q}}\left[\left(A,0,\Phi_{0}^{A}\right)\right]=-2\op{ind}\left(D_{B}^{+}\right)+\frac{1}{4}\left\langle c_{1}\left(S^{+}\right),c_{1}\left(S^{+}\right)\right\rangle -\frac{1}{4}\sigma\left(Z\right),
\]
with $\sigma\left(Z\right)$ denoting the signature of $Z$, $S^{+}$
the bundle $S_{+}^{TZ}$ from above, and $\textrm{ind}$ denoting
the complex index, namely the difference in complex dimensions, compare
\cite[S 3.4]{Cristofaro-Gardiner2013}.

The APS index theorem for spin-c Dirac operators now gives 
\[
\textrm{ind}(D_{B}^{+})=\frac{1}{8}\left\langle c_{1}\left(S^{+}\right),c_{1}\left(S^{+}\right)\right\rangle -\frac{1}{24}\int_{X}p_{1}+\frac{\eta(D_{A})}{2}.
\]
The APS signature theorem for the manifold $X$ with boundary also
gives 
\[
-\frac{1}{24}\int_{X}p_{1}=-\frac{1}{8}\sigma\left(Z\right)-\frac{1}{8}\eta_{Y},
\]
where $\eta_{Y}$ is the eta invariant of the odd signature operator
on $C^{\infty}\left(Y;T^{*}Y\oplus\mathbb{R}\right)$ sending 
\[
\left(a,f\right)\mapsto\left(\ast da-df,-d^{*}a\right).
\]
Combining the above we have 
\[
\textrm{gr}^{\mathbb{Q}}\left[\left(A,0,\Phi_{0}^{A}\right)\right]=-\eta(D_{A})+\frac{1}{4}\eta_{Y}.
\]
A reducible generator $\left[\mathfrak{a}_{k}\right]=\left[\left(A,0,\Phi_{k}^{A}\right)\right]\in\mathfrak{C}^{s}$
however has $\frac{1}{2}F_{A^{t}}=-d\mu$ and $\Phi_{k}^{A}$ the
$k$th eigenvector of $D_{A}$. Hence, 
\[
\textrm{gr}^{\mathbb{Q}}\left[\left(A,0,\Phi_{k}^{A}\right)\right]=2k+(-\eta\left(D_{A}\right)+\frac{1}{4}\eta_{Y})-2\op{sf}\lbrace D_{A_{s}}\rbrace_{0\le s\le1},
\]
where $D_{A_{s}}$ is a family of Dirac operators, associated to a
family of connections starting at the flat connection and ending at
one satisfying $\frac{1}{2}F_{A^{t}}=-d\mu.$ Hence, by interpreting
this spectral flow as an index through another application of Atiyah-Patodi-Singer
\cite[p. 95]{APSIII}, and applying \cite[eq. 4.3]{APSI} to compute
this index, we get 
\begin{align}
\textrm{gr}^{\mathbb{Q}}\left[\left(A,0,\Phi_{k}^{A}\right)\right] & =2k-\eta\left(D_{A}\right)+\frac{1}{4}\eta_{Y}-\frac{1}{2\pi^{2}}CS(A)\label{eq: Q grading reducible generator}
\end{align}
as the absolute grading of a reducible generator. 

The absolute grading of an irreducible generator $\left[\mathfrak{a}'\right]=\left(A',s,\Phi'\right)$,
$s\neq0$, is then given by 
\[
\textrm{gr}^{\mathbb{Q}}\left[\mathfrak{a}'\right]=\textrm{gr}^{\mathbb{Q}}\left[\mathfrak{a}_{0}\right]-2\op{sf}\left\{ \widehat{\mathcal{H}}_{\left(A_{\varepsilon},\Psi_{\varepsilon}\right)}\right\} _{0\leq\varepsilon\leq1}
\]
in terms of spectral flow of the Hessians \eqref{eq:Hessian} for
a path $\left(A_{\varepsilon},\Psi_{\varepsilon}\right)\in\mathcal{A}\left(Y,\mathfrak{s}\right)\times C^{\infty}\left(S\right)$,
$\varepsilon\in\left[0,1\right]$ starting at $\left[\mathfrak{a}_{0}\right]=\left[\left(A_{0},0\right)\right]$
and ending at $\left(A',s\Phi'\right)$. As above, we can interpret
this spectral flow as an index; this time, to compute the relevant
index, we need to apply (\cite[Thm. 3.10]{APSI}), which gives that
the above is equal to 
\[
\textrm{gr}^{\mathbb{Q}}\left[\mathfrak{a}'\right]=-\eta\left(\widehat{\mathcal{H}}_{\left(A,s\Phi'\right)}\right)+\frac{5}{4}\eta_{Y}-2\int_{Y\times\left[0,1\right]_{\varepsilon}}\rho_{0}.
\]
Here $\rho_{0}$ is the usual Atiyah-Singer integrand, namely the
local index density defined as the constant term in the small time
expansion of the local supertrace $\textrm{str}\left(e^{-t\mathcal{D}^{2}}\right)$
with 
\begin{equation}
\mathcal{D}=\begin{bmatrix} & -1\\
1
\end{bmatrix}\partial_{\varepsilon}+\begin{bmatrix} & 1\\
1
\end{bmatrix}\widehat{\mathcal{H}}_{\left(A_{\varepsilon},\Psi_{\varepsilon}\right)},\label{eq:Dirac on cylinder}
\end{equation}
and where $(A_{\epsilon},\Psi_{\epsilon})$ is the chosen path of
configurations. To compute the index density we choose a path of the
form 
\[
\left(A_{\varepsilon},\Psi_{\varepsilon}\right)=\begin{cases}
\left(A+2\varepsilon\left(A'-A\right),0\right); & 0\leq\varepsilon\leq\frac{1}{2},\\
\left(A',\left(2\varepsilon-1\right)\Psi\right); & \frac{1}{2}\leq\varepsilon\leq1.
\end{cases}
\]
On the interval $\left[0,\frac{1}{2}\right]$, the integral of the
local density is given by the usual local index theorem: as above,
we have 
\[
-2\int_{Y\times\left[0,\frac{1}{2}\right]_{\varepsilon}}\rho_{0}=-\frac{1}{2\pi^{2}}CS\left(A\right).
\]

On the other hand, for the calculation on $Y\times\left[\frac{1}{2},1\right]_{\varepsilon}$,
we have $\rho_{0}=0$. To see this, first note $\mathcal{D}^{2}=-\partial_{\varepsilon}^{2}+\widehat{\mathcal{H}}_{\left(A_{\varepsilon},\Psi_{\varepsilon}\right)}^{2}+\begin{bmatrix}-1\\
 & 1
\end{bmatrix}2M_{\Psi}$ gives 
\[
\textrm{str}\left(e^{-t\mathcal{D}^{2}}\right)=\textrm{tr}\left(e^{-t\left[-\partial_{\varepsilon}^{2}+\widehat{\mathcal{H}}_{\left(A_{\varepsilon},\Psi_{\varepsilon}\right)}^{2}-2M_{\Psi}\right]}-e^{-t\left[-\partial_{\varepsilon}^{2}+\widehat{\mathcal{H}}_{\left(A_{\varepsilon},\Psi_{\varepsilon}\right)}^{2}+2M_{\Psi}\right]}\right).
\]
Duhamel's principle then gives that the coefficients in the small
time heat kernel expansion of the difference above are of the form
\[
\begin{bmatrix}0 & 0 & *\\
0 & 0 & *\\*
* & * & 0
\end{bmatrix}
\]
with respect to the decomposition $iT^{*}Y\oplus\mathbb{R}\oplus S.$ 

Hence we have in summary: 
\[
\textrm{gr}^{\mathbb{Q}}\left[\mathfrak{a}\right]=\begin{cases}
2k-\eta\left(D_{A}\right)+\frac{1}{4}\eta_{Y}-\frac{1}{2\pi^{2}}CS\left(A\right); & \mathfrak{a}=\left(A,0,\Phi_{k}^{A}\right)\in\mathfrak{C}^{s},\\
-\eta\left(\widehat{\mathcal{H}}_{\left(A,s\Phi\right)}\right)+\frac{5}{4}\eta_{Y}-\frac{1}{2\pi^{2}}CS\left(A\right); & \mathfrak{a}=\left(A,s,\Phi\right)\in\mathfrak{C}^{o},\,s\neq0.
\end{cases}
\]

 \bibliographystyle{siam}
\bibliography{biblio}

\begin{thebibliography}{10}

\bibitem{Asaoka-Irie2016}
{\sc M.~Asaoka and K.~Irie}, {\em {A {$C^\infty$} closing lemma for
  {H}amiltonian diffeomorphisms of closed surfaces}}, Geom. Funct. Anal., 26
  (2016), pp.~1245--1254.

\bibitem{APSI}
{\sc M.~F. Atiyah, V.~K. Patodi, and I.~M. Singer}, {\em {Spectral asymmetry
  and {R}iemannian geometry. {I}}}, Math. Proc. Cambridge Philos. Soc., 77
  (1975), pp.~43--69.

\bibitem{APSIII}
\leavevmode\vrule height 2pt depth -1.6pt width 23pt, {\em {Spectral asymmetry
  and {R}iemannian geometry. {III}}}, Math. Proc. Cambridge Philos. Soc., 79
  (1976), pp.~71--99.

\bibitem{Barnes1901}
{\sc E.~W. {Barnes}}, {\em {The theory of the double gamma function.}},
  {Philos. Trans. R. Soc. Lond., Ser. A, Contain. Pap. Math. Phys. Character},
  196 (1901), pp.~265--387.

\bibitem{Beck-Robins2015}
{\sc M.~Beck and S.~Robins}, {\em {Computing the continuous discretely}},
  {Undergraduate Texts in Mathematics}, Springer, New York, second~ed., 2015.
\newblock Integer-point enumeration in polyhedra, With illustrations by David
  Austin.

\bibitem{Berline-Getzler-Vergne}
{\sc N.~Berline, E.~Getzler, and M.~Vergne}, {\em {Heat kernels and {D}irac
  operators}}, {Grundlehren Text Editions}, Springer-Verlag, Berlin, 2004.
\newblock Corrected reprint of the 1992 original.

\bibitem{Bismut-Freed-II}
{\sc J.-M. Bismut and D.~S. Freed}, {\em {The analysis of elliptic families.
  {II}. {D}irac operators, eta invariants, and the holonomy theorem}}, Comm.
  Math. Phys., 107 (1986), pp.~103--163.

\bibitem{Cristofaro-Gardiner2013}
{\sc D.~Cristofaro-Gardiner}, {\em {The absolute gradings on embedded contact
  homology and {S}eiberg-{W}itten {F}loer cohomology}}, Algebr. Geom. Topol.,
  13 (2013), pp.~2239--2260.

\bibitem{Cristofaro-Gardiner-Hutchings2016}
{\sc D.~Cristofaro-Gardiner and M.~Hutchings}, {\em {From one {R}eeb orbit to
  two}}, J. Differential Geom., 102 (2016), pp.~25--36.

\bibitem{Cristofaro-Gardiner-Hutchings-Pomerleano-2017}
{\sc D.~Cristofaro-Gardiner, M.~Hutchings, and D.~Pomerleano}, {\em {Torsion
  contact forms in three-dimensions have two or infinitely many Reeb orbits}},
  Preprint, available online at arxiv:1701.02262 (2017).

\bibitem{Cristofaro-Gardiner-Hutchings-Gripp2015}
{\sc D.~Cristofaro-Gardiner, M.~Hutchings, and V.~G.~B. Ramos}, {\em {The
  asymptotics of {ECH} capacities}}, Invent. Math., 199 (2015), pp.~187--214.

\bibitem{Cristofaro-Gardiner-Xueshan-Li-Stanley2018}
{\sc D.~Cristofaro-Gardiner, T.~Li, and R.~Stanley}, {\em {Irrational triangles
  with polynomial Ehrhart functions}},  (2018).

\bibitem{Dyatlov-Zworski2016}
{\sc S.~Dyatlov and M.~Zworski}, {\em {Dynamical zeta functions for {A}nosov
  flows via microlocal analysis}}, Ann. Sci. \'{E}c. Norm. Sup\'{e}r. (4), 49
  (2016), pp.~543--577.

\bibitem{Giulietti-Liverani-Pollicott}
{\sc P.~Giulietti, C.~Liverani, and M.~Pollicott}, {\em {Anosov flows and
  dynamical zeta functions}}, Ann. of Math. (2), 178 (2013), pp.~687--773.

\bibitem{Hutchings2011}
{\sc M.~Hutchings}, {\em {Quantitative embedded contact homology}}, J.
  Differential Geom., 88 (2011), pp.~231--266.

\bibitem{Hutchings2014lecture}
\leavevmode\vrule height 2pt depth -1.6pt width 23pt, {\em {Lecture notes on
  embedded contact homology}}, in {Contact and symplectic topology}, vol.~26 of
  {Bolyai Soc. Math. Stud.}, J\'{a}nos Bolyai Math. Soc., Budapest, 2014,
  pp.~389--484.

\bibitem{Irie2015}
{\sc K.~Irie}, {\em {Dense existence of periodic {R}eeb orbits and {ECH}
  spectral invariants}}, J. Mod. Dyn., 9 (2015), pp.~357--363.

\bibitem{Kronheimer-Mrowka}
{\sc P.~Kronheimer and T.~Mrowka}, {\em {Monopoles and three-manifolds}},
  vol.~10 of {New Mathematical Monographs}, Cambridge University Press,
  Cambridge, 2007.

\bibitem{Ruijsenaars2000}
{\sc S.~N.~M. Ruijsenaars}, {\em {On {B}arnes' multiple zeta and gamma
  functions}}, Adv. Math., 156 (2000), pp.~107--132.

\bibitem{Savale-Asmptotics}
{\sc N.~Savale}, {\em {Asymptotics of the {E}ta {I}nvariant}}, Comm. Math.
  Phys., 332 (2014), pp.~847--884.

\bibitem{Savale2017-Koszul}
{\sc N.~Savale}, {\em {Koszul complexes, {B}irkhoff normal form and the
  magnetic {D}irac operator}}, Anal. PDE, 10 (2017), pp.~1793--1844.

\bibitem{Savale-Gutzwiller}
\leavevmode\vrule height 2pt depth -1.6pt width 23pt, {\em {A {G}utzwiller type
  trace formula for the magnetic {D}irac operator}}, Geom. Funct. Anal., 28
  (2018), pp.~1420--1486.

\bibitem{Spreafico2009}
{\sc M.~Spreafico}, {\em {On the {B}arnes double zeta and {G}amma functions}},
  J. Number Theory, 129 (2009), pp.~2035--2063.

\bibitem{Taubes-Weinstein}
{\sc C.~H. Taubes}, {\em {The {S}eiberg-{W}itten equations and the {W}einstein
  conjecture}}, Geom. Topol., 11 (2007), pp.~2117--2202.

\bibitem{Taubes-ECH=HMI}
\leavevmode\vrule height 2pt depth -1.6pt width 23pt, {\em {Embedded contact
  homology and {S}eiberg-{W}itten {F}loer cohomology {I}}}, Geom. Topol., 14
  (2010), pp.~2497--2581.

\bibitem{Tsai-thesis-paper}
{\sc C.-J. Tsai}, {\em {Asymptotic spectral flow for {D}irac operations of
  disjoint {D}ehn twists}}, Asian J. Math., 18 (2014), pp.~633--685.

\bibitem{SunW-2018}
{\sc S.~Weifeng}, {\em {An estimate on energy of min-max Seiberg-Witten Floer
  generators }}, 2018.
\newblock arXiv:1801.02301.

\end{thebibliography}

\end{document}